\documentclass[11pt]{amsart} 
\usepackage{amssymb,amsfonts,graphicx,amsmath,amsthm,amstext,latexsym,amsfonts}
\usepackage{graphicx,epsfig,color}
\usepackage{bbm}
\usepackage{t1enc}
\usepackage{mathrsfs}
\usepackage{subfig}
\usepackage{hyperref}
\usepackage{a4wide}
\usepackage{graphicx}        % standard LaTeX graphics tool
\usepackage{tikz}
 \usetikzlibrary{arrows}
 \usetikzlibrary{calc}
\usepackage{pgfplots}
\usetikzlibrary{intersections}
\theoremstyle{plain}

\numberwithin{equation}{section}
\newtheorem{theorem}{Theorem}[section]
\newtheorem{proposition}[theorem]{Proposition}%[section]
\newtheorem{lemma}[theorem]{Lemma}%[section]
\newtheorem{remark}[theorem]{Remark}%[section]

\newenvironment{proofad1}{\removelastskip\par\medskip
\noindent{\textbf {Proof of Theorem \ref{maintheorem}}.}
\rm}{\penalty-20\null\hfill$\blacksquare$\par\medbreak} %%

\newenvironment{proofad2}{\removelastskip\par\medskip
\noindent{\textbf {Proof of Lemma \ref{appendixlemma}}.}
\rm}{\penalty-20\null\hfill$\blacksquare$\par\medbreak} %%

\usepackage{color}
\definecolor{darkred}{rgb}{0.8,0,0}
\definecolor{darkblue}{rgb}{0,0,0.7}
\definecolor{darkgreen}{rgb}{0,0.4,0}

\newcommand{\MMM}{\color{black}}
\newcommand{\KKK}{\color{black}}
\newcommand{\PPP}{\color{black}}

\newcommand{\eps}{\varepsilon}

\newcommand{\R}{{\mathbb R}}

\newcommand{\un}{{\rm 1\kern -2.5pt l}}

\def\eps{\varepsilon}
\def\R{{\mathbb R}}

\def\eps{\varepsilon}
\def\R{{\mathbb R}}

\def\argmin{\mathop{{\rm argmin}}\nolimits}

\renewcommand{\epsilon}{\varepsilon}
\newcommand{\beeq}{\begin{equation}}
\newcommand{\eneq}{\end{equation}}
\newcommand{\bear}{\begin{array}}
\newcommand{\enar}{\end{array}}
\newcommand{\bema}{\begin{displaymath}}
\newcommand{\enma}{\end{displaymath}}

\newcommand{\beea}{\begin{eqnarray}}
\newcommand{\enea}{\end{eqnarray}}

\newcommand{\lab}[1]{ \label{#1} }

\newcommand{\bC}{{\mathbbm{1}}}
\title[]{\PPP On the weighted inertia-energy approach to forced wave equations}
   \author[]{ Edoardo Mainini and Danilo Percivale}
\date{}

 \address{Universit\`{a} degli Studi di Genova, Dipartimento di   Ingegneria Meccanica, Energetica, $\quad$ Gestionale e dei Trasporti (DIME),
  Via all'Opera Pia, 15 - 16145 Genova Italy.}
  \email{ edoardo.mainini@unige.it; percivale@dime.unige.it}
\begin{document}
 \maketitle
\begin{abstract} \PPP
We show convergence of minimizers of weighted inertia-energy functionals to solutions of initial value problems for a class of nonlinear wave equations. 
The result is given for the nonhomogeneous case under a natural growth assumption on the forcing term. \KKK
%This provides an approximation result for the solution to the initial value problem for the linear wave 
%equation at least in average sense.
\end{abstract}
%%%%%%%%%%%
\begin{center}
%\large\sffamily DRAFT
\end{center}
%\tableofcontents
\begin{flushleft}
  {\bf AMS Classification Numbers (2010):\,} 49J45, 70H30
\end{flushleft}
\vskip0.5cm

\section{Introduction}

The variational approach by De Giorgi to second order initial value problems is based on the minimization of \textit{weighted inertia-energy} (WIE) functionals.
We consider the following nonlinear nonhomogeneous wave equation of defocusing type
\beeq\label{eqlim0}w''-\Delta w=- |w|^{r-2}w+f(t,x) \quad\mbox{ in $(0,+\infty)\times\mathbb R^N$,}\eneq
where $r>1$, along with the initial conditions $w(0,x)=w_0(x)$ and $w'(0,x)=w_1(x)$ in $\R^N$.
%\[\lab{eqlim2}\left\{\begin{array}{lll}
%&w''-\Delta w=-\dfrac{p}{2}|w|^{p-2}w+f \quad\mbox{ in $(0,+\infty)\times\mathbb R^N$}\\
%&\\
%&w(0,x)=w_0(x)\quad\mbox{ in $\mathbb R^N$}\\&\\& w'(0,x)=w_1(x)\quad\mbox{ in $\mathbb R^N$}.\end{array}
%\right.
%\]
We introduce the associated WIE functional
\beeq\lab{DG}
\mathcal I_\eps(w):=\int_0^{+\infty}\int_{\R^N}e^{-t/\eps}\left\{\frac{\eps^2}{2}|w''|^2+\frac12|\nabla w|^2+\frac1r\,|w|^r-fw\right\}\,dx\,dt,\qquad\eps>0,
\eneq
to be minimized
among a suitable class of competing functions  $w:[0,+\infty)\times\mathbb R^N\to\mathbb R$ satisfying the same given initial conditions $w(0,x)=w_0(x)$, $w'(0,x)=w_1(x)$.  Here, the prime $'$ symbol stands for time derivative, $\nabla, \Delta$ are the gradient and Laplacian in the spatial variables,  and $f$ is a given space-time depending  forcing term.

 \medskip

%where
%$w_0,\ w_1\in C^{\infty}_0(\R^N)$.\\
It was conjectured by De Giorgi in \cite{DG}, see also \cite{DG2}, that if $r=2k$ for a positive integer $k$, if  $f\equiv 0$ and $w_\eps$  is a minimizer of $\mathcal I_\eps$ subject to the initial conditions $w(0,x)=w_0(x)$, $w'(0,x)=w_1(x)$ in $\R^N$, with $w_0\in C^\infty_c(\R^N)$, $w_1\in C^\infty_c(\R^N)$, 
then $w_\eps(t,x)\to w(t,x)$ for a.e. $(t,x)\in (0,+\infty)\times \R^N$ as $\eps\to 0$ and $w$ is a solution of the initial value problem
%\beeq\lab{eqlim1}\left\{\begin{array}{ll}
%&w''-\Delta w=-w^{2k-1}\quad\mbox{ in $(0,+\infty)\times\mathbb R^N$,}\\
%&\\
%&w(0,x)=w_0(x),\;\;\; w'(0,x)=w_1(x) \quad\mbox{ in $\mathbb R^N$.}\end{array}
%\right.
%\eneq
\beeq\lab{eqlim2}\left\{\begin{array}{ll}
&w''-\Delta w=-|w|^{r-2}w\quad\mbox{ in $(0,+\infty)\times\mathbb R^N$,}\\
&\\
&w(0,x)=w_0(x),\;\;\; w'(0,x)=w_1(x)\quad\mbox{ in $\mathbb R^N$.}\end{array}
\right.
\eneq
This conjecture has been proved by Serra and Tilli in \cite[Theorem 1.1]{ST} in the following form:  if $r>1$, $f\equiv0$,  $(\eps_n)\subset(0,1)$ is a vanishing sequence, and if $w_{n}$ is a minimizer of $\mathcal I_{\eps_n}$ subject to the initial conditions $w(0,x)=w_0(x)\in H^1(\R^N)\cap L^p(\R^N)$, $w'(0,x)=w_1(x)\in H^1(\R^N)\cap L^p(\R^N)$, then, along a suitable subsequence, $w_n$ converge  a.e. in space-time to a solution of the initial value problem
\eqref{eqlim2}.
 %for equation \eqref{}, with the initial data $w_0, w_1$.
%for equation \eqref{eqlim2}, with initial data $w_0,w_1$.
The constructed solution is a weak solution satisfying the natural energy inequality, and the result by Serra and Tilli \cite{ST}
is thus recovering De Giorgi's conjecture if $r=2k$ as soon as such solutions are known to be unique, for instance in the subcritical regime $1< k< \frac{N}{N-2}$, see for instance \cite{IMM, Sh, Str, Str2}, that is, for every $k > 1$ if $N=2$, and for $k=2$  if $N=3$. 
 On the other hand, if $N\ge 3$, for supercritical  $k$, global uniqueness of weak solutions is not known in general  and this restricts the convergence result of \cite{ST} to subsequences as $\eps\to0$. 
   Therefore, uniqueness of variational solutions (i.e., obtained as limit of minimizers of the WIE functional) is still an open problem in the supercritical case, even for smooth compactly supported initial data. This problem is difficult due to the ill-posedness issues (see for instance \cite{L,T}) of  supercritical defocusing wave equations.
 % The supercritical wave equation
 
 \medskip
 
It is also possible to consider WIE functionals of the form
\[
\mathcal G_\eps(w):= \int_0^{+\infty}\int_{\R^N}e^{-t/\eps}\left\{\frac{\eps^2}{2}|w''|^2-f\, w\right\}\,dx\,dt+\int_0^{+\infty}e^{-t/\eps} W(w(t,\cdot))\,dt,\,\qquad\eps>0,
\] 
where the term
\begin{equation}\label{basicform}\int_{\R^N}\left(\frac12|\nabla w(t,x)|^2+\frac1r|w(t,x)|^r\right)\, dx\end{equation} in functional \eqref{DG}
has been replaced by a more general functional of the Calculus of Variations $ W(w(t,\cdot))$, in order to obtain a variational approximation scheme to hyperbolic PDEs having the following formal structure
 \beeq\lab{eqlim3}\left\{\begin{array}{ll}
&w''-D W(w(t,\cdot))=f(t,x)\quad\mbox{ in $(0,+\infty)\times\mathbb R^N$,}\\
&\\
&w(0,x)=w_0(x),\ w'(0,x)=w_1(x)\quad\mbox{ in $\mathbb R^N$},\end{array}
\right.
\eneq
where $D W$ is the Gateaux differential of $ W$. This has been achieved in \cite{ST2}, where the result of \cite{ST} has been extended to this generalized framework, still in the homogeneous case $f\equiv 0$.
%In this framework, the result of \cite{ST} has been extended in \cite{ST2} to the above more general case,   
 Nonhomogeneous equations have been investigated  in \cite{TT1}. More precisely, letting $f\in L^2((0,T);L^2(\mathbb R^N))$ for every $T>0$,  
 in \cite{TT1} it is shown that  there exists an approximating family $(f_\eps)_{\eps>0}$ of functions, converging  to $f$ as $\eps\to0$ in $L^2((0,T);L^2(\mathbb R^N))$  for every $T>0$, such that by letting $w_\eps$ minimize the WIE functional
  \[
\int_0^{+\infty}\int_{\R^N}e^{-t/\eps}\left\{\frac{\eps^2}{2}|w''|^2-f_\eps\, w\right\}\,dx\,dt+\int_0^{+\infty}e^{-t/\eps} W(w(t,\cdot))\,dt,\,\qquad\eps>0,
\]
subject to suitable initial conditions $w(0,x)=w_0(x)$, $w'(0,x)=w_1(x)$,  then up to subsequences there holds $w_\eps\to w$ a.e. in space-time, and $w$ is a solution to problem \eqref{eqlim3} featuring the same initial data. %Again, convergence of the whole family $(w_\eps)_\eps$ also holds if the limiting problem has a unique solution, as would be the case for instance for the linear nonhomogeneous wave equation $w''-\Delta w=f$ with $H^1(\R^N)$ initial data.
%\MMM condizioni sui dati iniziali? \KKK
 %  same result holds true af $\mathcal F_\eps$ by $\widetilde{\mathcal F}_\eps$, which is defined as $\mathcal F_\eps$ but with $f$ replaced by $f_\eps$.

\medskip

The results in \cite{TT1} make use of a specific construction of the approximating sequence $(f_\eps)_{\eps>0}$, \MMM with $f_\eps$ identically vanishing for small and for large $t$, \KKK not allowing for the choice $f_\eps\equiv f$. 
\MMM  On the other hand, our objective will be to obtain  convergence of minimizers of the original functionals $\mathcal G_\eps$ (which feature the given forcing term $f$) to the solution of the  associated initial value problem \eqref{eqlim3},   at least under some unavoidable growth assumptions on $f$ with respect to $t$ (see also Proposition \ref{sharp} below). \KKK
%we are interested in  considering the case $f_\eps\equiv f$.
 %\KKK Hence, our objective will be to investigate the problem of convergence of minimizers of the original functionals $\mathcal G_\eps$ to the solution of the  associated initial value problem \eqref{eqlim3}. 
 We will reach this result by considering the model case in which $W$ is in the form \eqref{basicform}, so that functional $\mathcal G_\eps$ is reduced to $\mathcal{I}_\eps$,  \MMM and with a sharp growth assumption on $f$, i.e., the Laplace transformability of the map  $t\mapsto \int_{\R^N}|f(t,x)|^2\,dx %\|f(t,\cdot)\|^2_{L^2(\mathbb R^N)}
$  (the result would be easier under stronger conditions like global boundedness of such a map, as assumed in our recent note \cite{MP} about the ode case). \KKK
% Here, we extend the results in \eqref{TT1} by showing that it is possible to obtain the same limit equation without introducing modifications to the forcing term $g$. Indeed.
In fact, we shall consider the more general WIE functional
\beeq\label{WIF}
\mathcal F_\eps(w):=\int_0^{+\infty}\int_{\R^N}e^{-t/\eps}\left\{\frac{\eps^2}{2}|w''|^2+\frac12|\nabla w|^2+\frac1r\,|w|^r-F(t,x,w)\right\}\,dx\,dt,\quad\eps>0,
\eneq
where  $r>1$ and  $((0,+\infty)\times\R^N\times \R)\ni(t,x,v)\mapsto F(t,x,v)$ is  gloabally Lipschitz  in the variable $v$ (detailed hypotheses on $F$ will be introduced in the next section). Letting  $G:={\partial F}/{\partial v}$,
 the associated equation is
\beeq\label{eqlim4}
w''-\Delta w=-|w|^{r-2}w+G(t,x,w)\quad\mbox{ in $(0,+\infty)\times\mathbb R^N$.}
\eneq
  Therefore, in our main result stated in Theorem \ref{maintheorem} below we will prove that if $w_\eps$ minimizes $\mathcal F_\eps$ subject to  the initial conditions $w(0,x)=w_0(x)\in H^1(\R^N)\cap L^r(\R^N)$, $w'(0,x)=w_1(x)\in H^1(\R^N)\cap L^r(\R^N)$, then along a vanishing sequence $w_\eps$ converge a.e. in space-time to a solution  of the initial value problem for equation \eqref{eqlim4} with same initial data $w_0,w_1$.

\medskip

Concerning further results in the literature about WIE approach to second order initial value problems, we mention the proof of the De Giorgi conjecture in a weaker version in \cite{S}. General semilinear equations, including dissipative equations, are treated in \cite{LS2, ST2}, see also \cite{TT2} where the nonhomogeneous case is considered. Moreover, the WIE approach applies for ode's, we refer to the approach to Lagrangian Mechanics in \cite{LS}, where the authors treat systems of the form $\mathbf y''-\nabla U(\mathbf y)=0$ for some given potential function $U$. The system of ode's $\mathbf y''=\mathbf f(t)$ is analyzed in our recent note \cite{MP}, which is in fact a basic application of the technique that we present in this paper. The proof we will develop  combines the estimates for hyperbolic equations from \cite{ST,ST2,TT1} with the approach from \cite{MP}.   

 \smallskip
 
 \subsection*{Plan of the paper} In Section \ref{mainsect} we introduce the notation and the assumptions of the theory, and then we state the main result.
 Section \ref{easysect} provides a couple of preliminary estimates. In Section \ref{exsect} we prove existence of minimizers and deduce an Euler-Lagrange equation for (a time rescaled version of) functional \eqref{WIF}. We also discuss the sharpness of our growth assumptions. In section \ref{proofsect} we prove the main result. An appendix includes further technical lemmas.

\section{Main result}\label{mainsect}
In this section we introduce the assumptions on the function $F$ appearing in \eqref{WIF} and we provide the functional framework for the main result. 

\subsection{Assumptions}\label{21}
 Let $(t,x,v)\mapsto F(t,x,v)$ be a
$
 L^1_{loc}((0,+\infty)\times \R^N\times\R)
$ function. We suppose 
 that 
\begin{equation}\label{C1}
\mbox{for a.e. $(t,x)\in(0,+\infty)\times\R^N$, the map $v\mapsto F(t,x,v)$ belongs to $C^1(\R)$},
\end{equation}
we define
\beeq   \lab{f}           G(t,x,v):=\frac{\partial F}{\partial v}(t,x,v)\quad\mbox{and}\quad  f(t,x):=\sup_{v\in\mathbb R}\left| G(t,x,v)\right|\eneq
and we assume that 
\beeq\lab{hyp0}f\in L^\infty((0,T);L^2(\R^N)) \quad \forall\ T>0.\eneq
Moreover we assume that there exists   $\eps_F\in(0,1/2)$   such that 
\begin{equation}\label{hyp1}
C_F:=\sup_{\eps\in(0,\eps_F)}\frac1\eps\int_0^{+\infty}\int_{\R^N}e^{-t/\eps}|F(t,x,0)|\,dx\,dt<+\infty
\end{equation}
and
\beeq\label{laplace}
\displaystyle K_F:=\frac1{\eps_F}\int_0^{+\infty} e^{-{t}/{(2\eps_F)}}\|f(t,\cdot)\|^2_{L^2(\R^N)}\,dt<+\infty.
\eneq

Some remarks on the above assumptions are in order.
 It is worth noticing that by \eqref{C1} the function $(t,x)\to G (t,x,v)$ is Lebesgue measurable for every $v\in \R$ and the function $v\to G (t,x,v)$ is continuous for a.e. $(t,x)\in(0,+\infty)\times\R^N$. 
Therefore for every measurable $u:(0,+\infty)\times\R^N\to \R$ the function $(t,x)\to G (t,x,u(t,x))$
is measurable and \eqref{f} entails
$$\left |G (t,x,u(t,x))\right|\le f(t,x)\;\;\; \mbox{for a.e.}\ (t,x)\in(0,+\infty)\times\R^N$$
whence by \eqref{hyp0} the function $(t,x)\to G (t,x,u(t,x))$ belongs to $L^\infty((0,T);L^2(\R^N))$ for every $T>0$.
  \eqref{hyp0} implies in particular that, for a.e. $(t,x)\in(0,+\infty)\times\R^N$, there holds $\sup_{v\in\mathbb R}\left|G(t,x,v)\right|<+\infty$ so that the map $v\mapsto F(t,x,v)$ is globally Lipschitz on $\R$, i.e.
\begin{equation}\label{lip}
|F(t,x,v_1)-F(t,x,v_2)|\le f(t,x)|v_1-v_2|\qquad \mbox{for every $v_1\in \R$ and every $v_2\in\R$},
\end{equation}
 and then $G$ coincides with the partial distributional derivative of $F$ with respect to its third variable.
 We notice that the main assumption \eqref{laplace} entails existence of the Laplace transform of the function $(0,+\infty)\ni t\mapsto\|f(t,\cdot)\|^2_{L^2(\R^N)}$ in the half space $\{z\in \mathbb C: \mathcal R e z > \frac{1}{2\eps_F}\}$.
 An assumption of this kind is necessary for the theory, as we discuss in Proposition \ref{sharp} later on.
  On the other hand, \eqref{hyp1} requires the existence of the Laplace transform of  the function $(0,+\infty)\ni t\mapsto\|F(t,\cdot,0)\|_{L^1(\R^N)}$ along with some control on its behavior for small $t$. 
 
 \smallskip
 
A relevant case is $F(t,x,v)=b(t,x)\,v$ so that the above assumptions reduce to  $b\in L^\infty((0,T);L^2(\R^N))$ for every $T>0$ along with the requirement that $t\mapsto\|b(t,\cdot)\|^2_{L^2(\R^N)}$ is Laplace transformable. With this choice of $F$, the equation associated with functional \eqref{WIF}, i.e., equation \eqref{eqlim4}, reduces to the form \eqref{eqlim0}.
More generally, we may let $F(t,x,v)=b(t,x)\psi(v)$ with $b$ as above, with the additional assumption that $\psi$ is a $C^1$ function such that $\psi'$ is bounded over $\R$  and $\psi(0)=0$. This includes, for instance, the classical sine-Gordon equation and the variable mass sine-Gordon equation \cite{K}, if $\psi(v)=\cos v -1$.

\subsection{Rescaled energy functionals}

Let $r>1$, let us introduce the Banach space
 $$H:=H^1(\R^N)\cap L^r(\R^N),$$
with dual space denoted by $H'$, and for every $v\in H$ let
\[W(v):=\frac12\|\nabla v\|^2_{L^2(\mathbb R^N)}+\frac {1}r\| v\|^r_{L^r(\mathbb R^N)}.
\] \MMM
We denote by $DW$ the Gateaux derivative of $W$ and we let
\[
(DW(v),u):=\int_{\mathbb R^N}\nabla v\cdot\nabla u+|v|^{r-2}vu,\,\qquad v\in H,\;\; u\in H.
\] \KKK
We define now
$$\displaystyle H^2_{loc}([0, +\infty); L^2(\R^N)):=\displaystyle\bigcap_{T> 0}H^2((0,T);L^2(\R^N))$$
and for every $w_0,\ w_1\in H$ we let
\beeq\lab{Ueps}
\mathcal U_{\eps}^0:=\{u\in \mathcal U:  u(0,\cdot)=w_0(\cdot),\, u'(0,\cdot)=\eps w_1(\cdot)\},\quad\eps>0,
\eneq
where
\beeq\lab{U}
\mathcal U:=\left\{u\in H^2_{loc}([0,+\infty); L^2(\R^N)): \int_{0}^{+\infty}e^{-t}\|u''(t,\cdot)\|_{L^2(\R^N)}^2\,dt<+\infty\right\}.
\eneq
%\MMM Here, for a function in $u\in\mathcal U$, matching the initial conditions (i.e., being in $\mathcal U_\eps^0$)  is understood  as follows: the $AC([0,T];L^2(\R^N))$ representatives of $u$ and $u'$ have value $w_0$ and $\eps w_1$ at $t=0$, respectively. 
%We also notice that the function $w_0(x)+t\eps w_1(x)$ belongs to $\mathcal U_\eps^0$. \KKK

%and where, for every $u\in H^2_{loc}([0, +\infty);L^2(\mathbb R^N))$ and  for every $t\in \R^+$,  we have set $u(t)(x):= u(t,x)$ for a.e. $x\in \R^N$.\\
We recall from \cite{ST2,TT1}
the following elementary estimate. \MMM 
For every $u\in\mathcal U$ we have
\beeq\label{stimacon0}
\int_0^{+\infty}e^{-t}\|u(t,\cdot)\|^2_{L^2(\R^N)}\,dt\le 2\|u(0,\cdot)\|_{L^2(\R^N)}^2+4\int_0^{+\infty}e^{-t}\|u'(t,\cdot)\|^2_{L^2(\R^N)}\,dt
\eneq
and the same if $u$ is replaced with $u'$.
As a consequence, if $\eps<1$, \KKK
 for every $u\in \mathcal U_{\eps}^0$ there hold
\begin{equation}\label{'0}
 \int_{0}^{+\infty}e^{-t}\|u'(t,\cdot)\|_{L^2(\R^N)}^2\,dt\le 2\|w_1\|_{L^2(\mathbb R^N)}^2+4 \int_{0}^{+\infty}e^{-t}\|u''(t,\cdot)\|_{L^2(\R^N)}^2\,dt,\end{equation}
\begin{equation}\label{''0}\begin{aligned}
 \int_{0}^{+\infty}e^{-t}\|u(t,\cdot)\|_{L^2(\R^N)}^2\,dt&\le 2\|w_0\|_{L^2(\mathbb R^N)}^2\\&\quad+8\|w_1\|_{L^2(\mathbb R^N)}^2
+16 \int_{0}^{+\infty}e^{-t}\|u''(t,\cdot)\|_{L^2(\R^N)}^2\,dt.
\end{aligned}\end{equation}
Under the validity of the assumptions of Subsection \ref{21}, we set now for every $\eps\in(0,\eps_F)$, where $\eps_F\in(0,1/2)$ is the number appearing in \eqref{hyp1} and \eqref{laplace}, and for every $ u\in \mathcal U_{\eps}^0$
\beeq\lab{Phi}
\Phi_{\eps}(u):= \int_0^{+\infty}e^{-t}\int_{\R^N}F(\eps t,x,u(t,x))\,dx\,dt.
\eneq
By  \eqref{hyp1} and by Cauchy-Schwarz inequality we get
\[
\begin{aligned}
&\left|\Phi_{\eps}(u)\right|\le C_F+ \int_0^{+\infty}\int_{\R^N}e^{-t}|F(\eps t,x, u(t,x))-F(\eps t,x,0)|\,dx\,dt\\
&\qquad\le C_F+ \int_0^{+\infty}\int_{\R^N}e^{-t}f(\eps t,x) u(t,x)\,dx\,dt\\&\qquad\le C_F+\left(\int_0^{+\infty}e^{-t}\|f(\eps t,\cdot)\|^2_{L^2(\R^N)}\right)^{\frac12}\left( \int_{0}^{+\infty}e^{-t}\|u(t,\cdot)\|_{L^2(\R^N)}^2\,dt\right)^{\frac12},
\end{aligned}
\]
therefore by \eqref{''0},  by Young inequality and by a very simple estimate that follows from \eqref{hyp0} and \eqref{laplace} and whose proof is postponed to Lemma \ref{quicklemma} below, we get the existence of a constant $K_F^*$, only depending on $f$ and $\eps_F$, such that for every $\eps\in(0,\eps_F)$
\begin{equation}\label{below}
\begin{aligned}
&\left|\Phi_{\eps}(u)\right|\le C_F+\sqrt{K_F^*}\left( \int_{0}^{+\infty}e^{-t}\|u(t,\cdot)\|_{L^2(\R^N)}^2\,dt\right)^{\frac12}\\
&\qquad \le C_F+16K_F^*+\frac1{32}\|w_0\|^2_{L^2(\R^N)}+\frac18\|w_1\|^2_{L^2(\R^N)}+\frac1{4} \int_{0}^{+\infty}e^{-t}\|u''(t,\cdot)\|_{L^2(\R^N)}^2\,dt.
\end{aligned}
\end{equation}
% defined by
%\[
%\mathcal F_\eps(u):=\int_0^{+\infty}\int_{\R^N}e^{-t/\eps}\left\{\frac{\eps^2}{2}|w''(t,x)|^2+\frac12|\nabla w(t,x)|^2+\frac1r|w(t,x)|^r-F( t,x,w(t,x))\right\}\,dt\,dx
%\]
%and the rescaled functional
Therefore we may define 
%$\mathcal J_{\eps,\lambda}: \to \R$
\beeq\lab{Jeps}
\mathcal J_{\eps}(u):=\displaystyle\int_0^{+\infty}e^{-t}\left\{\frac1{2\eps^2}\|u''(t,\cdot)\|_{L^2(\R^N)}^2+W(u(t,\cdot))\right\}dt- \Phi_{\eps}(u)\eneq
and we see that for $\eps\in(0,\eps_F)$ the functionals $\mathcal J_\eps$ are well defined and finite for every $u\in \mathcal U_{\eps}^0$.  %The two are related by $\mathcal F_\eps(w)=\eps\mathcal J_\eps(u)$ as soon as $u(t,x)=w(\eps t,x)$ as directly seen by the change of variable $t\mapsto \eps t$.

%The associated nonlinear wave equation is
%\[
%w''=\Delta w-|w|^{r-2}w+G(t,x,w)
%\]
%where $G=\partial F/\partial v$.

%For every $v\in L^r(\R^N)$, $w\in L^r(\R^N)$ with $\nabla v\in L^2(\R^N)$ and $\nabla w\in L^2(\R^N)$, it is convenient to introduce the functional
%\[
%W(v):=\frac12\|\nabla v\|^2_{L^2(\mathbb R^N)}+\frac1r\| v\|^r_{L^r(\mathbb R^N)},
%\]
%and the duality $(DW(v),w)$,  defined by
%\[
 %(DW(v),w):=\int_{\R^N}\nabla v(x)\cdot\nabla w(x)\,dx+\int_{\R^N}| v(x)|^{r-2} v(x) w(x)\,dx.
%\]

\subsection{Main result}
We are now in a position to state our main result.
\begin{theorem}\label{maintheorem} Assume that the hypotheses 
of {\rm Subsection \ref{21}}
hold true and let 
$w_0,\ w_1\in H$. Then, for every $\eps\in (0,\eps_F)$ there exists a solution $u_\eps$ to problem $$\min\{\mathcal J_\eps(u):u\in\mathcal U_\eps^0\}.$$ 
%For every $\eps\in(0,\eps_F)$, let  $u_\eps$ be a solution to problem \eqref{min}
\MMM Let $(\eps_n)_{n\in\mathbb N}\subset (0,\eps_F)$ be a vanishing sequence, let $u_{\eps_n}$ minimize $\mathcal J_{\eps_n}$ over $\mathcal U_{\eps_n}^0$ and let $w_{\eps_n}(t,x):=u_{\eps_n}(\frac{t}{\eps_n},x)$ for every $n\in\mathbb N$.
%Letting $w_\eps(t,x):=u_\eps(\frac{t}{\eps},x)$, if  $(\eps_n)_{n\in\mathbb N}\subset (0,\eps_F)$ is a vanishing sequence,
 Then, there exist a subsequence $(\eps_{n_k})_{k\in\mathbb N}$ and a function $w\in H^1_{loc}((0,+\infty)\times\mathbb R^N)$ such that 
 $w_{\eps_{n_k}}\to w$   %a.e. in $(0,+\infty)\times\R^N$ and
  weakly in $H^1_{loc}((0,+\infty)\times\mathbb R^N)$ as $k\to+\infty$
  and such that $w$ satisfies
$$\mbox{$w\in W^{2,\infty}((0,T);H')\cap H^1((0,T); L^2(\R^N))\cap L^2((0,T);H^1(\R^N))\cap L^r((0,T)\times\R^N)\;$  $\forall \,T>0$},$$

%$$\mbox{   for every $T>0$},$$

%\\&\\
$$w''-\Delta w=-|w|^{r-2}w+G(t,x,w)\quad\mbox{ in $\;\mathcal D'((0,+\infty)\times\R^N)$},$$

  $$  \mbox{$w(0,x)=w_0(x)$ for a.e. $x\in\R^N$,  $\;\;w'(0,x)=w_1(x)$  for a.e. $x	\in\R^N$},$$
%\beeq\lab{limeq}\left\{\begin{array}{ll}
%&w''=\Delta w- |w|^{r-2}w+\frac{\partial F}{\partial v}\quad \mbox{in}\ \mathcal D'(\R^+\times\R^N)\\
%&\\
%&w(0,x)=w_0(x),\ w'(0,x)=w_1(x)\ \mbox{a.e. in}\ \R^N .\\
%\end{array}\right.
%\eneq
\MMM
along with the following energy inequality
\beeq\label{energyinequality}
 \sqrt{\mathcal E(t)}\le \sqrt{\mathcal E(0)}+\sqrt{\frac t2\int_0^t\|f(s,\cdot)\|^2_{L^2(\R^N)}\,ds}\qquad \mbox{for a.e. $t>0$,}
\eneq
where $\mathcal E(t):=\dfrac12\|w'(t,\cdot)\|^2_{L^2(\R^N)}+W(w(t,\cdot))$ and $f$ is defined by \eqref{f}.
\KKK
\end{theorem}

%\begin{remark}\rm
%In the proof we will also obtain boundedness of the family of minimizers $(w_\eps)_{\eps\in(0,\eps_F)}$    in \MMM $W^{2,\infty}((0,T);H')$, \KKK for every $T>0$. In particular, the constructed solution $w$ belongs to $W^{2,\infty}((0,T);H')$ for every $T>0$. % and  we have $w_{\eps_n}(0,\cdot)\to w$ and $w'_{\eps_n}(0,\cdot)\to w'$ in $H'$ as $n\to+\infty$.
%\end{remark}

\begin{remark}\label{change} \rm It is worth noticing that functional $\mathcal J_\eps$ is a rescaling of the original functional $\mathcal F_\eps$ from \eqref{WIF}, and  if $u_\eps\in \argmin_{\mathcal U_\eps^0} \mathcal J_{\eps}$,  then 
%by setting
%\beeq
%\Psi_{\eps}(w):= \int_0^{+\infty}e^{-t/\eps}\int_{\R^N}F(t,x,w(t,x))\,dx\,dt,
%\eneq
 $w_\eps$ minimizes indeed $\mathcal F_\eps$
%\beeq\lab{eFps}
%\mathcal F_{\eps}(u):=\displaystyle\int_0^{+\infty}e^{-t/\eps}\left\{\frac{\eps^2}{2}\|w''(t)\|_{L^2(\R^N)}^2+W(w(t))\right\}dt- \Psi_{\eps}(w)\eneq
over $\mathcal V_\eps^0:=\{w\in\mathcal V_\eps: w(0,\cdot)=w_0(\cdot),\,w'(0,\cdot)=w_1(\cdot)\}$, where
\[
 \mathcal V_\eps:=\left\{w\in H^2_{loc}([0,+\infty); L^2(\R^N)): \int_{0}^{+\infty}e^{-t/\eps}\|w''(t)\|_{L^2(\R^N)}^2\,dt<+\infty\right\}.
\]
\end{remark}

\begin{remark}\rm If we change the definitions of $H$ and $W$  letting $H= H^1(\R^N)$ \KKK and $W(u):=\frac12\|\nabla u\|^2_{L^2(\R^N)}$, and accordingly we redefine functional \eqref{Jeps}, with the same proof we obtain the analogous convergence result of rescaled minimizers of functional $\mathcal J_\eps$ to a solution of the initial value problem for the  equation $w''-\Delta w=G(t,x,w)$.
We also expect that our result extends to the case in which $W$ is a more general functional as in \cite{ST2,TT1} thus covering equations of the form $w''-DW(w)=G(t,x,w)$.  
\end{remark}

\section{Some preliminary estimates}\label{easysect} In this section the assumptions of Subsection \ref{21} are understood to hold. In particular,
useful consequences of the assumptions \eqref{hyp0} and \eqref{laplace} are obtained in the following two lemmas.
\begin{lemma}\label{quicklemma}
For every $\eps\in(0,\eps_F)$ there holds
\[
\int_0^{+\infty} e^{-t}\|f(\eps t,\cdot)\|^2_{L^2(\R^N)}\,dt\le K_F^*%:=K_F+\|f\|^2_{L^\infty((0,1);L^2(\R^N))},
\]
and more generally for  every $\eps\in(0,\eps_F)$ and every $\alpha\ge0$ there holds
\[
\int_0^{+\infty} t^\alpha e^{-t}\|f(\eps t,\cdot)\|^2_{L^2(\R^N)}\,dt\le K_F^*(\alpha),%:=\int_0^{+\infty} \frac{t^{\alpha}e^{-t/\eps_F}}{\eps_F^{\alpha+1}}\|f(t,\cdot)\|^2_{L^2(\R^N)}\,dt+\Gamma(\alpha+1)\,\|f\|^2_{L^\infty((0,\alpha+1);L^2(\R^N))},
\]
where
\[
K_F^*(\alpha):=\int_0^{+\infty} \frac{t^{\alpha}e^{-t/\eps_F}}{\eps_F^{\alpha+1}}\|f(t,\cdot)\|^2_{L^2(\R^N)}\,dt+\Gamma(\alpha+1)\,\|f\|^2_{L^\infty((0,\alpha+1);L^2(\R^N))}
\]
and $K_F^*:=K_F(0)$. Here, $\Gamma$ denotes the Euler Gamma function.
\end{lemma}
\begin{proof}
By changing variables, since $\eps_F<1$ and since the map $(0,t)\ni\eps\mapsto \eps^{-1}e^{-t/\eps}$ is increasing for every $t>0$, we deduce
\[\begin{aligned}
&\int_0^{+\infty} e^{-t}\|f(\eps t,\cdot)\|^2_{L^2(\R^N)}\,dt\le\int_0^1\eps^{-1}e^{-t/\eps}\|f(t,\cdot)\|^2_{L^2(\R^N)}\,dt+\int_1^{+\infty}\eps^{-1}e^{-t/\eps}\|f(t,\cdot)\|^2_{L^2(\R^N)}\,dt\\&\qquad\qquad\le (1-e^{-1/\eps})\,\|f\|^2_{L^\infty((0,1);L^2(\R^N))}+\int_0^{+\infty}\eps_F^{-1}e^{-t/\eps_F}\|f(t,\cdot)\|^2_{L^2(\R^N)}\,dt
\end{aligned}\]
which is the desired estimate, the right hand side being finite thanks to \eqref{hyp0} and \eqref{laplace}. The more general estimate is obtained with the same argument, taking into account that $\eps_F<1$, that the map $(0,t/(\alpha+1))\ni\eps\mapsto \eps^{-1-\alpha}t^\alpha e^{-t/\eps}$ is increasing for every $t>0$ and that by definition of the Gamma function $\int_0^{+\infty}t^{\alpha} e^{-t/\eps}\,dt=\eps^{\alpha+1}\Gamma(\alpha+1)$. We notice that $K_F^*(\alpha)$ is finite for every $\alpha\ge 0$ thanks to \eqref{hyp0} and \eqref{laplace}.
\end{proof}

\KKK

\begin{lemma}\label{Linfty}
 For every  $s> 0$ \KKK and every $\eps\in(0,\eps_F)$, let 
 \beeq\label{Theta}Q_{\eps}(s):=\int_0^{+\infty}\eps^{-2}te^{-t/\eps}\|f(t+s,\cdot)\|^2_{L^2(\R^N)}\,dt.\eneq
  Then, for every $\eps\in(0,\eps_F)$, we have $Q_\eps\in L^\infty(0,T)$ for every $T>0$ and in particular
 \[
 \|Q_\eps\|_{L^\infty(0,T)}\le M_F(T):= \|f\|^2_{L^\infty((0,2T+1);L^2(\R^N))}+{\eps_F^{-2}}\,e^{T/\eps_F} \int_0^{+\infty}te^{-t/\eps_F}\|f(t,\cdot)\|^2_{L^2(\R^N)}\,dt.
 \]
\end{lemma}
\begin{proof}
We shall use the notations $N_\eps(t):=\eps^{-2}te^{-t/\eps}$ and $b_f(t):=\|f(t,\cdot)\|^2_{L^2(\R^N)}$. We notice that  $\int_0^{+\infty}N_\eps(t)\,dt=1$.
For every $\eps\in(0,\eps_F)$,  every $s>0$ \KKK and every $T>0$,  we have
\[\begin{aligned}
\int_0^{+\infty}N_\eps(t)b_f(t+s)\,dt&=\int_0^{T+1} N_\eps(t)b_f(t+s)\,dt+\int_{T+1}^{+\infty}N_\eps(t)b_f(t+s)\,dt\\&\le\|b_f(\cdot+s)\|_{L^\infty(0, T+1)}+\int_{T+1}^{+\infty}N_{\eps_F}(t)b_f(t+s)\,dt,
\end{aligned}\]
where the last inequality is due to the fact that $(0,t/2)\ni\eps\mapsto N_\eps(t)$ is an increasing function for every $t>0$, and here we have $\eps_F<1/2<(T+1)/2$. Therefore for every $s>0$ \KKK and every $T>0$
\[\begin{aligned}
Q_\eps(s)&\le \|b_f\|_{L^\infty(0, T+1+s)}
+\int_{T+s+1}^{+\infty}N_{\eps_F}(t-s)b_f(t)\,dt\\&\le
\|b_f\|_{L^\infty(0, T+1+s)}+{\eps_F^{-2}}\,e^{s/\eps_F}\int_0^{+\infty}te^{-t/\eps_F}b_f(t)\,dt
\end{aligned}\]
and the right hand side is finite due to the assumptions on $f$. This shows that $Q_\eps(s)$ is well defined for every  $s> 0$ \KKK and moreover we deduce that for every $T>0$
\[
\|Q_\eps\|_{L^\infty(0,T)}\le \|b_f\|_{L^\infty(0,2T+1)}+{\eps_F^{-2}}\,e^{T/\eps_F} \int_0^{+\infty}te^{-t/\eps_F}b_f(t)\,dt
\]
thus proving the result.
\end{proof}

\MMM
The function $Q_\eps$ from Lemma \ref{Linfty} is defined as a convolution and we may deduce its behavior as $\eps\to0$. Indeed, 
if $T>0$ we have $Q_\eps\to b_f$ in $L^1(0,T)$, where $b_f(t):=\|f(t,\cdot)\|^2_{L^2(\R^N)}$. This is a consequence of the following approximation lemma.
\begin{lemma}\label{Lebesgue}  Let $N_\eps(t):=\eps^{-2}te^{-t/\eps}$, $t\ge 0$, $\eps>0$. Let $b\in L^1_{loc}(0,+\infty)$ and $\delta>0$ be such that $\int_0^{+\infty} e^{-t/\delta}|b(t)|\,dt<+\infty$, and suppose $b\in L^1(0,T)$ for every $T>0$. Then $N_\eps(\cdot)b(\cdot+s)\in L^1(0,+\infty)$ for every $s>0$ and every $\eps\in(0,\delta)$. Moreover, letting $b_\eps(s):=\int_0^{+\infty}N_\eps(t)b(t+s)\,dt$, for every $T>0$ we have $b_\eps\in L^1(0,T)$ and
$b_\eps\to b$ in $L^1(0,T)$ as $\eps\to 0$. %If in addition $b\in C^0([0,+\infty))$, then we also have $b_\eps (t)\to b(t)$ for every $t\ge 0$.
\end{lemma}
\begin{proof}
For every $s>0$ and every $\eps\in(0,\delta)$ we have 
\[
\int_0^{+\infty}N_\eps(t)|b(s+t)|\,dt=\eps^{-2}\int_s^{+\infty}(t-s)e^{-(t-s)/\eps}|b(t)|\,dt\le \eps^{-2}e^{s/\eps}\int_0^{+\infty}te^{-t/\eps}|b(t)|\,dt<+\infty,\]
and if $T>0$,
estimating in analogous way we see that
\[\begin{aligned}\int_0^T|b_\eps(s)|\,ds=\int_0^T\frac{e^{s/\eps}}{\eps^2}\int_s^{+\infty}(t-s)e^{-t/\eps}|b(t)|\,dt\,ds\le  \frac{Te^{T/\eps}}{\eps^2}\int_0^{+\infty}te^{-t/\eps}|b(t)|\,dt<+\infty\end{aligned}\]
so that indeed $N_\eps(\cdot)b(\cdot+s)\in L^1(0,+\infty)$ and $b_\eps\in L^1(0,T)$.
\MMM For every $\eps\in (0,\delta)$ and $z\in\R$ let  $$\tilde N_{\eps}(z):=\delta^{-2}(\delta-\eps)^2 e^{-z/\delta} N_{\eps}(-z)\bC_{(-\infty,0]}(z),\qquad \tilde b_{\eps}(z):=\delta^{2}(\delta-\eps)^{-2} e^{-z/\delta}b(z)\bC_{(0, +\infty)}(z).$$ It is readily seen that the family $(\tilde N_{\eps})_{\eps\in (0,\delta)}$ is an approximate identity on $\R$  (see \cite{A,B}) and that $\tilde b_{\eps}\to \tilde b$ in $L^1(\mathbb R)$ as $\eps\to0$, where $\tilde b(z):=e^{-z/\delta}b(z)\bC_{(0, +\infty)}(z).$ Therefore, $\tilde N_\eps\ast\tilde b\to\tilde b$ in $L^1(\R)$ as $\eps\to0$, thus Young's convolution inequality entails   $\tilde N_{\eps}*\tilde b_{\eps}\to \tilde b$ in $L^1(\mathbb R)$. But for every $s> 0$ we have
\[\displaystyle(\tilde N_{\eps}*\tilde b_{\eps})(s)=\int_{-\infty}^{+\infty}\tilde N_{\eps}(s-z)\tilde b_{\eps}(z)\,dz
 %= \int_{-\infty}^{+\infty}\tilde N_{\eps}(-t)\tilde b_{\eps}(s+t)\,dt
 =e^{-s/\delta}\int_{0}^{+\infty} N_{\eps}(t) b(s+t)\,dt=e^{-s/\delta}b_\eps(s)=: b^*_{\eps}(s),\\
\]
hence $b^*_{\eps}\to \tilde b$ in $L^1(0,+\infty)$ as $\eps\to 0$, which entails
$b_{\eps}\to b$ in $L^1(0,T)$ for every $T>0$ as claimed. \KKK
\end{proof}
%{\color{blue} For every $\eps\in (0,\delta)$ let  $$\tilde N_{\eps}(s):=\delta^{-2}(\delta-\eps)^2 e^{-s/\delta} N_{\eps}(-s)\mathbf 1_{(-\infty,0]}(s),\ \tilde b_{\eps}(s):=\delta^{2}(\delta-\eps)^{-2} e^{-s/\delta}b(s)\mathbf 1_{(0, +\infty)}(s).$$ It is readily seen that the family $(\tilde N_{\eps})_{\eps\in (0,\delta)}$ is a smooth approximate identity (see \cite{A}, \cite{B}) and that $\tilde b_{\eps}\to \tilde b$ in $L^1(\mathbb R)$ where $\tilde b(s)=e^{-s/\delta}b(s)\mathbf 1_{(0, +\infty)}(s).$ Therefore $\tilde N_{\eps}*\tilde b_{\eps}\to \tilde b$ in $L^1(\mathbb R)$ and for every $s\ge 0$ we get
%\beeq\begin{array}{ll}&\displaystyle(\tilde N_{\eps}*\tilde b_{\eps})(s)=\int_{-\infty}^{+\infty}\tilde N_{\eps}(s-\tau)\tilde b_{\eps}(\tau)\,d\tau=\\
%&\\
%&\displaystyle = \int_{-\infty}^{+\infty}\tilde N_{\eps}(-t)\tilde b_{\eps}(s+t)\,dt=e^{-s/\delta}\int_{0}^{+\infty} N_{\eps}(t) b(s+t)\,dt:= b^*_{\eps}(s)\\
%\end{array}
%\eneq
%hence $b^*_{\eps}\to \tilde b$ in $L^1(0,+\infty)$ which entails
%$b_{\eps}\to b$ in $L^1(0,T)$ for every $T>0$ as claimed.}
\KKK

\section{Existence of minimizers and Euler-Lagrange equation}\label{exsect}
Also in this section the assumptions of Subsection \ref{21} are understood to hold.
The first step towards the proof of the main theorem  is to minimize functional $\mathcal J_\eps$ from \eqref{Jeps} in the class 
$\mathcal U_\eps$ from \eqref{Ueps}-\eqref{U}.
%\[
%\mathcal U_\eps^0:=\{u\in \mathcal U:  u(0)=w_0,\, u'(0)=\eps w_1\},
%\]
%where
%\[
%\mathcal U:=\left\{u\in H^2_{loc}([0,+\infty);L^2(\mathbb R^N)): \int_{0}^{+\infty}\int_{\mathbb R^N}e^{-t}|u''(t,x)|^2\,dx\,dt<+\infty\right\}.
%\]

\begin{lemma}\label{exist} Let $\eps\in (0,\eps_F)$.
There exists a  solution to the minimization problem 
\begin{equation}\label{min}
\min\left\{ \mathcal J_\eps(u): u\in\mathcal U^0_\eps\right\}.
\end{equation}
Moreover, %$u_\eps \in H^1((0,T)\times\mathbb R^N)$ for every $T>0$ and
 there exists an explicit constant $\bar C$ (only depending on $F$, $w_0$ and $w_1$) such that if $u_\eps$ is solution to \eqref{min} there holds
\begin{equation}\label{basicenergy}
\int_0^{+\infty}e^{-t}\left\{\frac1{2\eps^2}\|u_{\eps}''(t)\|_{L^2(\R^N)}^2+W(u_{\eps}(t))\right\}\,dt\le \bar C,
\end{equation}
\end{lemma}
\begin{proof}

 Since $\eps<\eps_F<1$, estimate \eqref{below} shows that $\mathcal J_\eps$ is  bounded from below over $\mathcal U_\eps^0$ by a constant that depends only on $F, w_0, w_1$.

% which, after using Cauchy-Schwarz and Young inequality along with \eqref{''0}, entails 
%\[\begin{aligned}
%&\int_0^{+\infty}\int_{\mathbb R^N}e^{-t}\left\{\frac1{2\eps^2}|u_n''(t,x)|^2+\frac12|\nabla u_n(t,x)|^2\right\}\,dx\,dt\le \frac1{32}\|w_0\|_{L^2(\mathbb R^N)}^2+\frac18\|w_1\|_{L^2(\mathbb R^N)}^2\\&\quad+\frac14\int_0^{+\infty}\int_{\mathbb R^N}e^{-t}|u_n''(t,x)|^2\,dx\,dt+\frac{1}{16}\int_0^{+\infty}\int_{\mathbb R^N}e^{-t}|f(\eps t,x)|^2\,dx\,dt+\mathcal J_\eps(w_0+t\eps w_1)+1.
%\end{aligned}\]
%Therefore, since $\eps<1$ and since \eqref{hyp} holds, 
%\[\begin{aligned}
%\int_0^{+\infty}\int_{\mathbb R^N}e^{-t}\left\{\frac1{4\eps^2}|u_n''(t,x)|^2+\frac12|\nabla u_n(t,x)|^2\right\}\,dx\,dt&\le \frac1{32}\|w_0\|_{L^2(\mathbb R^N)}^2+\frac18\|w_1\|_{L^2(\mathbb R^N)}^2\\&\quad +16C +\mathcal J_\eps(w_0+t\eps w_1)+1.
%\end{aligned}\]
\MMM We  notice that the function $w_0(x)+t\eps w_1(x)$ belongs to $\mathcal U_\eps^0$.
Moreover, \KKK since $\eps<\eps_F<1/2$, by Cauchy-Schwarz inequality, by \eqref{hyp1} and by Lemma \ref{quicklemma} we have similarly as in estimate  \eqref{below}
\[\begin{aligned}
\mathcal J_\eps(w_0+\eps t w_1)&=\int_0^{+\infty}\int_{\mathbb R^N}e^{-t}\left\{\frac12|\nabla w_0+t\eps \nabla w_1|^2+\frac1r|w_0+\eps tw_1|^r- F(\eps t,x,w_0+\eps t w_1)\right\}\\
& \le(\|\nabla w_0\|^2_{L^2(\R^N)}+\|\nabla w_1\|^2_{L^2(\R^N)})+2^{r-1}\Gamma(r+1)( \| w_0\|^r_{L^r(\R^N)}+\| w_1\|^r_{L^r(\R^N)})\\
&\qquad\qquad+C_F+\sqrt{K_F^*}\,(2\|w_0\|^2_{L^2(\R^N)}+\|w_1\|^2_{L^2(\R^N)})^{1/2},
\end{aligned}\]
 showing in particular that $J_\eps(w_0+t\eps w_1)<+\infty$ so that indeed $\inf\left\{ \mathcal J_\eps(u): u\in\mathcal U^0_\eps\right\}\in \R$.

Let  $(u_n)_{n\in\mathbb N}\in \mathcal U_\eps^0$ be a minimizing sequence for \eqref{min}. It is not restrictive to assume that the minimizing sequence satisfies $\mathcal J_\eps(u_n)\le 1+\mathcal J_\eps(w_0+t\eps w_1)$ for every $n$, and this implies, along with the previous estimate of  $J_\eps(w_0+t\eps w_1)$ and \eqref{below},  that there exists an explicit constant $\bar C$, only depending on $F,w_0, w_1$, such that for every $n\in\mathbb N$
\[\begin{aligned}
&\int_0^{+\infty}\int_{\R^N}e^{-t}\left\{\frac1{2\eps^2}|u_n''(t,x)|^2+\frac12|\nabla u_n(t,x)|^2+\frac1r|u_n(t,x)|^r\right\}\,dt\,dx\\&\qquad\qquad\le \frac12\bar C+\frac1{4}\int_0^{+\infty}\int_{\R^N}e^{-t}|u_n''(t,x)|^2\,dx\,dt,
\end{aligned}
\]
and since $\eps<\eps_F<1$ we conclude that 
\begin{equation}\label{unif}\int_0^{+\infty}\int_{\R^N}e^{-t}\left\{\frac1{2\eps^2}|u_n''(t,x)|^2+\frac12|\nabla u_n(t,x)|^2+\frac1r|u_n(t,x)|^r\right\}\,dt\,dx\le \bar C.\end{equation}
Therefore, by taking also \eqref{'0}-\eqref{''0} into account, the sequence $(e^{-t/2}u_n)_{n\in\mathbb N}$ is uniformly bounded in $H^2((0,+\infty);L^2(\R^N))$, in particular it converges weakly to some suitable $v$ in $L^2((0,+\infty)\times\R^N)$ up to extracting a subsequence. Moreover,  if  $Q_{T_R}:=(0,T)\times B_R$, where $B_R$ is a ball of radius $R$ in $\mathbb R^N$,  we obtain that $u_n$ $u_n'$ and $\nabla u_n$ enjoy uniform bounds in $L^{2}(Q_{T,R})$, thus up to further subsequences, $u_n$ converges to a suitable $u$ strongly in $L^{2}(Q_{T,R})$ (and then pointwise) by Rellich theorem. The pointwise convergence allows to identify the limit, i.e., $v=e^{-t/2} u$.
By \eqref{unif} and the weak lower semicontinuity of $L^p$ norms, $u$ satisfies \eqref{basicenergy}.

We further notice that by \eqref{lip} and Lemma \ref{quicklemma} we have
\[\begin{aligned}
&\int_0^{+\infty}\int_{\R^N}e^{-t}|F(\eps t,x,u_n(t,x))-F(\eps t,x, u(t,x))|\,dx\,dt\\&\;\;\; \le\int_0^{+\infty}\int_{\R^N}e^{-t}f(\eps t,x)|u_n(x,t)-u(x,t)|\,dx\,dt\\&\;\;\;
= \iint_{Q_{R,T}}e^{-t}f(\eps t,x)|u_n(x,t)-u(x,t)|\,dx\,dt+\iint_{Q_{R,T}^c} e^{-t}f(\eps t,x)|u_n(x,t)-u(x,t)|\,dx\,dt\\
&\;\;\;\le \sqrt{K_F^*}\left(\iint_{Q_{R,T}}|u_n(t,x)-u(t,x)|^2\,dx\,dt\right)^{\frac12}+\iint_{Q_{R,T}^c} e^{-t}f(\eps t,x)|u_n(x,t)-u(x,t)|\,dx\,dt
\end{aligned}
\]
for every $T>0$, $R>0$, $n\in\mathbb N$.
We split the last integral in an integral over $(T,+\infty)\times \R^N$ plus an integral over $(0,T)\times B_R^c$, and we estimate as
\[\begin{aligned}
&\int_T^{+\infty} \int_{\R^N} e^{-t}f(\eps t,x)|u_n(x,t)-u(x,t)|\,dx\,dt\le \sqrt{K_F^*}\int_T^{+\infty}e^{-t}\|u_n(t,\cdot)-u(t,\cdot)\|_{L^2(\R^N)}\,dt\\&\qquad\le
\sqrt{K_F^*}\, e^{-T/2}\left(\int_0^{+\infty}\int_{\R^N}e^{-t}|u_n(t,x)-u(t,x)|^2\,dx\,dt\right)^\frac12\le \sqrt {C_*\, K_F^*\, e^{-T}},
\end{aligned}\]
where the latter inequality is due to \eqref{''0}, since $u_n$ satisfies \eqref{unif} and $u$ satisfies \eqref{basicenergy}, by letting $C_*=8\|w_0\|^2_{L^2(\R^N)}+32\|w_1\|^2_{L^2(\R^N)}+128\bar C$. Moreover,  
\[\begin{aligned}
\int_0^{T} \int_{B_R^c} e^{-t}f(\eps t,x)|u_n(x,t)-u(x,t)|\,dx\,dt\le \sqrt{C_*}
\left(\int_0^T\int_{B_R^c}e^{-t}f(\eps t,x)^2\,dx\,dt\right)^\frac12
\end{aligned}\]
and the right hand side vanishes as $R\to+\infty$, since Lemma \ref{quicklemma} implies that $e^{-t}f(\eps t,x)^2$ belongs to $L^1((0,T)\times\mathbb R^N)$.
Therefore, since $u_n$ strongly converge to $ u$ in $L^2(Q_{R,T})$, we get
\[\begin{aligned}&\limsup_{n\to+\infty} \int_0^{+\infty}\int_{\R^N}e^{-t}|F(\eps t,x,u_n(t,x))-F(\eps t,x, u(t,x))|\,dx\,dt\\&\qquad\qquad\le \sqrt{C_*}
\left(\int_0^T\int_{B_R^c}e^{-t}f(\eps t,x)^2\,dx\,dt\right)^\frac12 + \sqrt{ C_*\, K_F^*\, e^{-T}},\end{aligned}\]
and by taking the limit as $R\to+\infty$ and then the limit as $T\to+\infty$ we get
\[
\lim_{n\to+\infty}\int_0^{+\infty}\int_{\R^N}e^{-t}F(\eps t,x,u_n(t,x))\,dx\,dt=\int_0^{+\infty}\int_{\R^N}e^{-t}F(\eps t,x,u(t,x))\,dx\,dt.
\]
The latter convergence, along with the semicontinuity of $L^p$ norms, allows to conclude that
\[
\mathcal J_\eps(u)\le \liminf_{n\to+\infty}\mathcal J_\eps(u_n),
\]
thus $u$ is solution to problem \eqref{min}. 

Finally, assuming that 
 $u_\eps $ solves problem \eqref{min}, since $\mathcal J_\eps(u_\eps)\le\mathcal J_\eps(w_0+t\eps w_1) $ and since we can apply \eqref{below} with $\bar u=u_\eps$, by repeating the arguments in the first part of the proof we conclude that $u_\eps$ itself satifies \eqref{basicenergy}. 
 \end{proof}
 
 \MMM
 In the limit as $\eps\to 0$ we also have the following estimate for minimizers.
 \begin{lemma}\label{newlemma}
 For every $\eps\in(0,\eps_F)$, let  $u_\eps$ be a solution to problem \eqref{min}. Then
\[
\limsup_{\eps\to 0} \int_0^{+\infty}e^{-t}\left\{\frac1{2\eps^2}\|u_{\eps}''(t)\|_{L^2(\R^N)}^2+W(u_{\eps}(t))\right\}\,dt\le W(w_0).
\]
\end{lemma}
\begin{proof}
 The estimate $\mathcal J_\eps(u_\eps)\le\mathcal J_\eps(w_0+t\eps w_1) $ entails
 \begin{equation}\label{erre}
 \int_0^{+\infty}e^{-t}\left\{\frac1{2\eps^2}\|u_{\eps}''(t)\|_{L^2(\R^N)}^2+W(u_{\eps}(t))\right\}\,dt\le \int_0^{+\infty}e^{-t}W(w_0+\eps tw_1)\,dt +\mathcal R_\eps, \end{equation}
 where $\mathcal R_\eps:=|\Phi_\eps(u_\eps)-\Phi_\eps(w_0+\eps t w_1)|$, and where $\Phi_\eps$ is defined by \eqref{Phi}. 
 We claim that $ \mathcal R_\eps$ goes to $0$ as $\eps\to0$. Indeed, 
 by \eqref{lip} and Cauchy-Schwarz  inequality we get
 \[\begin{aligned}
 &\mathcal R_\eps\le \left(\int_0^{+\infty}e^{-t}\|f(\eps t,\cdot)\|^2_{L^2(\R^N)}\,dt\right)^{\frac12}\left(\int_0^{+\infty}2e^{-t}\|u_\eps(t,\cdot)-w_0\|^2_{L^2(\R^N)}\,dt+4\eps^2\|w_1\|^2_{L^2(\R^N)}\right)^{\frac12}
 \end{aligned}
 \]
 and  by applying \eqref{stimacon0}  twice along with \eqref{basicenergy} we obtain
 \[\begin{aligned}
 \int_0^{+\infty}e^{-t}\|u_\eps(t,\cdot)-w_0\|^2_{L^2(\R^N)}\,dt&\le 8\eps^2\|w_1\|^2_{L^2(\R^N)}+16\int_0^{+\infty}e^{-t}\|u_\eps''(t,\cdot)\|^2_{L^2(\R^N)}\,dt\\&\le 8\eps^2\|w_1\|^2_{L^2(\R^N)}+32 \eps^2\bar C.\end{aligned}
 \]
 Thanks to the latter two estimates and to Lemma \ref{quicklemma}, the claim is proved. Since it is clear that
 \[
 \lim_{\eps\to0} \int_0^{+\infty}e^{-t}W(w_0+\eps tw_1)\,dt=\int_0^{+\infty}e^{-t}W(w_0)\,dt=W(w_0),
 \]
 from \eqref{erre} we obtain the result.
 \end{proof}
%\begin{remark}\label{change}\rm By noticing that $\mathcal J_\eps$ and $\mathcal F_\eps$ are related by $\mathcal F_\eps(w)=\eps\mathcal J_\eps(u)$ if  $u(t,x)=w(\eps t,x)$, it is clear that if $u_\eps$ is a solution to problem \eqref{min}, then $w_\eps(t,x):=u_\eps(t/\eps,x)$ minimizes 
%\beeq\lab{Feps}\begin{array}{ll}
%&\displaystyle\mathcal F_\eps(w):=\int_0^{+\infty}\int_{\R^N}e^{-t/\eps}\left\{\frac{\eps^2}{2}|w''(t,x)|^2+\frac12|\nabla w(t,x)|^2\right\}\,dt\,dx+\\
%&\\
%&\displaystyle +\int_0^{+\infty}\int_{\R^N}e^{-t/\eps}\left\{\frac1r|w(t,x)|^r-F( t,x,w(t,x))\right\}\,dt\,dx
%\end{array}
%\eneq

%over $\mathcal V_\eps^0:=\{w\in\mathcal V_\eps: w(0,\cdot)=w_0,\,w'(0,\cdot)=w_1\}$, where
%\[
% \mathcal V_\eps:=\left\{w\in H^2_{loc}((0,+\infty);L^2(\mathbb R^N)): \int_{0}^{+\infty}\int_{\mathbb R^N}e^{-t/\eps}|w''(t,x)|^2\,dx\,dt<+\infty\right\}.
%\]
%\end{remark}

\KKK

\MMM
A key property of minimizers that is needed for the proof of the main theorem is provided in the next lemma, which builds on the results of \cite{ST,ST2, TT1}, exploiting the estimates of the {\it approximate energy} introduced therein.
The proof requires some technical preliminaries and is therefore postponed to the appendix.   \KKK
\begin{lemma} \label{appendixlemma}
\MMM For every $\eps\in(0,\eps_F)$, let  $u_\eps$ be a solution to problem \eqref{min}. \KKK
 There  exists a constant $Q$, \KKK not depending on $\eps$ but only depending on $w_0,w_1$ and $F$, such that for every $t\ge0$
\begin{equation}\label{tt}
\begin{aligned}
&\int_t^{+\infty}(s-t)e^{-(s-t)} W(u_\eps(s,\cdot))\,ds\le E_\eps(t)\le  Q+(28+4\eps t)\,\int_0^{\eps t}Q_\eps(s)\,ds\end{aligned}
\end{equation}
\KKK
where %letting $K_\eps(t):=\frac1{2\eps^2}\|u_\eps'(t,\cdot)\|^2_{L^2(\mathbb R^N)}$, 
the approximate energy $E_\eps$ is defined by
\beeq\label{approxenergy}E_\eps(t):=\frac1{2\eps^2}\|u_\eps'(t,\cdot)\|^2_{L^2(\mathbb R^N)}+\int_t^{+\infty}(s-t)e^{-(s-t)}W(u_\eps(s,\cdot))\,ds,\eneq
and where $Q_\eps$ is the function defined in \eqref{Theta}.
   \MMM 
Moreover, for every $t\ge 0$ there holds
\beeq\label{newinappendixlemma}
\limsup_{\eps\to0}\sqrt{E_\eps(t/\eps)}\le \left(\frac12\|w_1\|_{L^2(\R^N)}^{2}+W(w_0)\right)^{\frac12}+\sqrt{t/2}\,\left(\int_0^t\|f(s,\cdot)\|^2_{L^2(\R^N)}\,ds\right)^{\frac12}.
\eneq
\KKK
\end{lemma}

Making use of Lemma \ref{appendixlemma} we also obtain the following

\begin{lemma}\label{lemma1} Let $\eps\in(0,\eps_F)$ and let $u_\eps$ be a solution to problem \eqref{min}. Let  $\varphi\in C^{\infty}_c((0,+\infty))$.
Then there exists a constant $\bar Q$ (only depending on $w_0,w_1$ and $F$) such that we have the following three estimates
\[
\int_0^{+\infty} |\xi(t)|  e^{-t} \left(\int_{\R^N}|\nabla u_\eps(t,x)|^2\,dx\right)^{\frac{1}{2}}\,dt\le 2\int_0^{\infty} \left(\bar Q+(28+4\eps t)\int_0^{\eps t}Q_\eps(s)\,ds\right)\,|\varphi(t)|\,dt,
\]
\[
\int_0^{+\infty} |\xi(t)|  e^{-t} \left(\int_{\R^N}| u_\eps(t,x)|^r\,dx\right)^{\frac{r-1}{r}}\,dt\le r\int_0^{\infty} \left(\bar Q+(28+4\eps t)\int_0^{\eps t}Q_\eps(s)\,ds\right)\,|\varphi(t)|\,dt,
\]
\[
\int_0^{+\infty} |\xi(t)|  e^{-t}\|f(\eps t,\cdot)\|_{L^2(\mathbb R^N)}\,dt\le\int_0^{\infty}|\varphi(t)|\,\sqrt{Q_\eps(\eps t)}\,dt,
\]
all the integrals being finite,
where $Q_\eps$ is the function defined by \eqref{Theta}
%\[
%\int_0^{+\infty} |\xi(t)| \sqrt{W(u_\eps(t,\cdot))} e^{-t}\,dt\le \frac{1}{2}\int_0^{\infty} (1+Q_1+Q_2\eps^2 s^2)|\varphi(s)|\,ds
%\]
 and $\xi:[0,+\infty)\to\mathbb R$ is defined by
\begin{equation}\lab{xi}\xi(t):=\int_0^t (t-s)e^s\varphi (s)\,ds.\end{equation}
\end{lemma}
\begin{proof}
The right hand sides of the three estimates are finite since $\varphi\in C^{\infty}_c((0,+\infty))$ and thanks to Lemma \ref{Linfty}.

Let us prove the second estimate. Since $(rx)^{(r-1)/r}\le r(1+x)$ for every $r>1$ and every $x\ge 0$,
 by taking \eqref{xi} into account and by exploiting Fubini Theorem  we have
 \begin{equation*}\label{0T}\begin{aligned}
&\displaystyle \int_0^\infty |\xi(t)|e^{-t}\left(\int_{\R^N}| u_\eps(t,x)|^r\,dx\right)^{\frac{r-1}{r}}\,dt
\le\displaystyle \int_0^\infty |\xi(t)|e^{-t}\left(rW(u_\eps(t,\cdot))\right)^{\frac{r-1}{r}}\,dt\\
&\quad \le\displaystyle r\int_0^\infty |\xi(t)|e^{-t}(1+W(u_\eps(t,\cdot)))\,dt
=
r\int_0^\infty \left |\int_0^t (t-s)e^s\varphi (s)\,ds\right | e^{-t}(1+W(u_\eps(t,\cdot)))\,dt\\
&\quad\le
\displaystyle r\int_0^\infty \int_0^t (t-s)|\varphi (s)|e^{-t+s}(1+W(u_\eps(t,\cdot)))\,ds\,dt\\&\quad
%\\&\qquad=p\int_0^\infty |\varphi (s)|\,ds\int_s^{\infty} (t-s)e^{-t+s}(1+W(u_\eps(t,\cdot)))\,dt\\
 = r\int_0^\infty |\varphi (s)|\,ds\int_s^{\infty} (t-s)e^{-(t-s)}(1+W(u_\eps(t,\cdot)))\,dt
\end{aligned}\end{equation*}
%
%\begin{equation*}\label{0T}\begin{aligned}
%&\displaystyle \int_0^\infty |\xi(t)|e^{-t}\sqrt {W(u_\eps(t,\cdot))}\,dt=\int_0^\infty \left |\int_0^t (t-s)e^s\varphi (s)\,ds\right | e^{-t}\sqrt {W(u_\eps(t,\cdot))}\,dt\\
%&\qquad\le
%\displaystyle \int_0^\infty \int_0^t (t-s)|\varphi (s)|e^{-t+s}\sqrt {W(u_\eps(t,\cdot))}\,ds\,dt\\&\qquad=\int_0^\infty |\varphi (s)|\,ds\int_s^{\infty} (t-s)e^{-t+s}\sqrt {W(u_\eps(t,\cdot))}\,dt\\
%&\displaystyle\qquad \le \frac{1}{2}\int_0^\infty |\varphi (s)|\,ds\int_s^{\infty} (t-s)e^{-(t-s)}(1+W(u_\eps(t,\cdot)))\,dt
%\end{aligned}\end{equation*}
and then the second estimate in the statement follows from  \eqref{tt} in Lemma \ref{appendixlemma}, letting $\bar Q:=Q+1$ where $Q$ is the constant therein.
 The first estimate is proven in the same way. 
In order to check the third estimate, we notice that by Jensen inequality
\[\begin{aligned}
\int_0^{+\infty} |\xi(t)|  e^{-t}\|f(\eps t,\cdot)\|_{L^2(\mathbb R^N)}\,dt&\le \int_0^{+\infty}\left(\int_0^t(t-s)e^s|\varphi(s)|\,ds\right) e^{-t}\|f(\eps t,\cdot)\|_{L^2(\mathbb R^N)}\,dt\\&=\int_0^{+\infty}|\varphi(s)|\left(\int_s^{+\infty}(t-s)e^{-(t-s)}\|f(\eps t,\cdot)\|_{L^2(\R^N)}\,dt\right)\,ds\\&\le 
\int_0^{+\infty}|\varphi(s)|\left(\int_s^{+\infty}(t-s)e^{-(t-s)}\|f(\eps t,\cdot)\|^2_{L^2(\R^N)}\,dt\right)^\frac12\,ds\\
\end{aligned}\]
which is the desired result since
\[
Q_\eps(\eps s)=\int_0^{+\infty}\eps^{-2}te^{-t/\eps}\|f(t+\eps s,\cdot)\|^2_{L^2(\R^N)}\,ds=\int_s^{+\infty}(t-s)e^{-(t-s)}\|f(\eps t,\cdot)\|^2_{L^2(\R^N)}\,dt\]
follows by change of variables.
\end{proof}\KKK

\MMM
We state a first order minimality condition satisfied by a solution to \eqref{min}.
We stress that it is obtained by taking advantage of Lemma \ref{appendixlemma} and Lemma \ref{lemma1}, which exploit the results in the appendix that are also based on a first variation argument obtained by a different perturbation of the minimizer.  \KKK

\begin{lemma}\label{lemmaEL}
Let $\eps\in(0,\eps_F)$ and let $u_\eps$ be a solution to problem \eqref{min}.
  Let $\varphi\in C^\infty_c((0,+\infty))$ and let $\xi$  be defined by \eqref{xi}. \MMM Let $h\in H$. \KKK
  Then there holds
   \begin{equation}\label{EL}\begin{aligned}
&\frac1{\eps^2}\int_0^{+\infty}\int_{\R^N}e^{-t}u''_\eps(t,x)\xi''(t)h(x)\,dx\,dt+\int_0^{+\infty}e^{-t}\xi(t)\,(DW(u_\eps(t,\cdot)),h)\,dx\,dt\\&\qquad-\int_0^{+\infty}\int_{\R^N}e^{-t}\,G(\eps t,x,u_\eps(t,x))\,\xi(t)h(x)\,dx\,dt=0.
\end{aligned}\end{equation}
\end{lemma}
\begin{proof}
Let us first consider the mapping $$(-1,1)\ni\delta\mapsto\frac1r\int_{0}^{+\infty}\int_{\R^N}e^{-t}|u_\eps(t,x)+\delta\xi(t)h(x)|^r\,dx\,dt.$$ 
Here, the integral is finite for every $\delta\in(-1,1)$, because both $e^{-t}|u_\eps(t,x)|^r$ and $e^{-t}|\xi(t)h(x)|^r$ are in $L^1((0,+\infty)\times\R^N)$ thanks to the fact that $u_\eps$ satisfies  \eqref{basicenergy} and since $|\xi(t)|\le C^*_\varphi t$ for every $t>0$, where $C^*_\varphi:=\int_0^{+\infty} e^s|\varphi(s)|\,ds$.
We claim that such a mapping is differentiable at $\delta=0$ and more precisely that
\[\begin{aligned}
&\lim_{\delta\to 0}\frac1{\delta r}\int_{0}^{+\infty}\int_{\R^N}e^{-t}\left(|u_\eps(t,x)+\delta\xi(t)h(x)|^r-|u_\eps(t,x)|^r\right)\,dx\,dt\\&\qquad\qquad=\int_0^{+\infty}\int_{\R^N}e^{-t}\xi(t)h(x)|u_\eps(t,x)|^{r-2}u_\eps(t,x)\,dx\,dt
\end{aligned}\]
Indeed, while pointwise a.e. convergence of the integrands is clear (as $r>1$), by Lagrange theorem we have for every $0<|\delta|<1$ 
\[\begin{aligned}
&e^{-t}\left|\frac{|u_\eps(t,x)+\delta\xi(t)h(x)|^r-|u_\eps(t,x)|^r}{\delta r}\right|\le e^{-t}(|u_\eps(t,x)|+|\xi(t)h(x)|)^{r-1}|\xi(t)h(x)|\\&\qquad\qquad\le e^{-t} (1\vee 2^{r-2})\left(|u_\eps(t,x)|^{r-1}+|\xi(t)h(x)|^{r-1}\right)\,|\xi(t)h(x)|,
\end{aligned}\]
where the right hand side is integrable, because $e^{-t}|\xi(t) h(x)|^r\in L^1((0,+\infty)\times\R^N)$ as already observed and
\[
\int_{0}^{+\infty}\int_{\R^N}e^{-t} |u_\eps(t,x)|^{r-1} |\xi(t)h(x)|\,dx\,dt\le \|h\|_{L^r(\R^N)}\int_0^{+\infty}e^{-t}|\xi(t)|\|u_\eps(t,\cdot)\|^{r-1}_{L^r(\R^N)}\,dt<+\infty
\]
 thanks to Lemma \ref{lemma1}. Therefore, the claim follows by dominated convergence. The same argument shows that
 \[\begin{aligned}
 &\lim_{\delta\to 0}\frac1{2\delta }\int_{0}^{+\infty}\int_{\R^N}e^{-t}\left(|\nabla u_\eps(t,x)+\delta\xi(t)\nabla h(x)|^2-|\nabla u_\eps(t,x)|^2\right)\,dx\,dt\\&\qquad\qquad=\int_0^{+\infty}\int_{\R^N}e^{-t}\xi(t)\nabla h(x)\cdot\nabla u_\eps(t,x)\,dx\,dt.
 \end{aligned}\]
 On the other hand we have
 \[\begin{aligned}
& \frac1{2\eps^2}\int_0^{+\infty}\int_{\R^N}e^{-t} |u_\eps''(t,x)+\delta\xi''(t)h(x)|^2\,dx\,dt= \frac1{2\eps^2}\int_0^{+\infty}\int_{\R^N}e^{-t} |u_\eps''(t,x))|^2\,dx\,dt\\&\qquad+ \frac\delta{\eps^2}\int_0^{+\infty}\int_{\R^N}e^{-t} u_\eps''(t,x)\xi''(t)h(x)\,dx\,dt+ \frac{\delta^2}{2\eps^2}\int_0^{+\infty}\int_{\R^N}e^{-t} |\xi''(t)h(x)|^2\,dx\,dt,
\end{aligned} \]
 where all the integrals are finite thanks to Cauchy-Schwarz inequality, to \eqref{basicenergy} and to the fact that $\xi''(t)=e^t\varphi(t)$ so that $\xi''\in C^\infty_c((0,+\infty))$. Thus 
  \[\begin{aligned}
 \frac{d}{d\delta}{\Bigg{|}_{\delta=0}}\int_0^{+\infty}\int_{\R^N}e^{-t}|u_\eps''(t,x)+\delta\xi''(t)h(x)|^2\,dx\,dt=2\int_0^{+\infty}\int_{\R^N}e^{-t}\xi''(t) h(x) u_\eps''(t,x)\,dx\,dt.
 \end{aligned}\]
Moreover,
the mapping $\delta\mapsto \int_0^{+\infty}\int_{\R^N}e^{-t}F(\eps t,x, u_\eps(t,x)+\delta\xi(t)h(x))\,dx\,dt$ is differentiable at $\delta=0$ with derivative equal to 
\[\int_0^{+\infty}\int_{\R^N}e^{-t}G(\eps t,x,u_\eps(t,x))\,\xi(t)h(x)\,dx\,dt.\]
This can  be seen by dominated convergence as well.
Indeed, \eqref{C1} and \eqref{f} imply that 
\[
\lim_{\delta\to 0}\frac{F(\eps t,x, u_\eps(t,x)+\delta\xi(t)h(x))-F(\eps t,x,u_\eps(t,x))}{\delta}=G(\eps t,x,u_\eps(t,x))\,\xi(t)h(x)
\]
for a.e. $(t,x)\in(0,+\infty)\times\R^N$, and in order to find a dominating function we notice that  \eqref{lip} implies 
\[e^{-t}\,\frac{|F(\eps t,x, u_\eps(t,x)+\delta\xi(t)h(x))-F(\eps t,x,u_\eps(t,x))|}{|\delta|}\le e^{-t}f(\eps t,x)|\xi(t)||h(x)|
\]
where the right hand side is integrable over $(0,+\infty)\times\R^N$,  since by \eqref{laplace}, having $|\xi(t)|\le C^*_\varphi t$ for every $t>0$, we obtain
\[\begin{aligned}
&\int_0^{+\infty}\int_{\R^N}e^{-t}f(\eps t,x)|\xi(t)|h(x)|\,dx\,dt\le\|h\|_{L^2(\R^N)}\left(\int_{0}^{+\infty}e^{-t}|\xi(t)|^2\|f(\eps t,\cdot)\|^2_{L^2(\R^N)}\right)^{\frac12}<+\infty. 
\end{aligned}\]

We conclude that the mapping $(-1,1)\ni\delta\mapsto \mathcal J_\eps(u_\eps+\delta\xi h)$ is differentiable at $\delta=0$. Since $\xi(0)=\xi'(0)=0$, we have $u_\eps+\delta\xi h\in\mathcal U_\eps^0$ for every $\delta\in(-1,1)$, and since $u_\eps$ is a minimizer for $\mathcal J_\eps$ over $\mathcal U_\eps^0$ we deduce that \eqref{EL} holds. 
\end{proof}

We close this section by showing that our main result is not achievable without a growth assumption like \eqref{laplace}, because the energy functional $\mathcal J_\eps$ may easily get unbounded from below over $\mathcal U_\eps^0$. We consider for simplicity the case $F(t,x,v)=b(t,x)\,v$, so that the assumptions on $F$ from Subsection \ref{21} reduce to the following: $b\in L^\infty((0,T);L^2(\R^N))$ for every $T>0$, and the map $t\mapsto\|b(t,\cdot)\|^2_{L^2(\R^N)}$ is Laplace transformable. 

\begin{proposition}\label{sharp}
Let $\eps>0$. 
Suppose that $F(t,x,v)=\eta(t)\varphi(x)\, v$, where $\varphi\in L^2(\R^N)$ is nonnegative  and $\eta\in L^1_{loc}(0,+\infty)$ is nonnegative and such that $\eta\in L^\infty(0,T)$ for every $T>0$. \MMM Let $w_1\in H$ be nonnegative, let $w_0\in H$ and suppose that $\int_{\R^n}\varphi(x)w_0(x)\,dx>0$. \KKK
Assume that $\eta$ is not Laplace transformable, i.e., $\int_0^{+\infty}e^{-t/\delta}\eta(t)\,dt=+\infty$ for every $\delta>0$. 
Then  $\mathcal J_\eps$ is unbounded from below over $\mathcal U_\eps^0$. 
\end{proposition}
\begin{proof}
We equivalently prove that functional $\mathcal F_\eps$ is unbounded from below over $\mathcal V_\eps^0$ (see Remark \ref{change}).
With the given choice of $F$, the function $f$ from \eqref{f} is given by $f(t,x)=\eta(t)\varphi(x)$, and 
the assumptions \eqref{C1}, \eqref{hyp0} and \eqref{hyp1}  are trivially satisfied.  On the other hand, by Jensen inequality the map $t\mapsto\|f(t,\cdot)\|^2_{L^2(\R^N)}$ is  not Laplace transformable so that \eqref{laplace} does not hold.
For every $n\in\mathbb N$, let 
 $w_n(t,x):=w_0(x)\zeta_n(t)+tw_1(x)\zeta_n(t)$, where $\zeta_n(t):=\zeta(2^{-n}t)$ and $\zeta:[0,+\infty)\to[0,1]$ is smooth, decreasing, compactly supported and satisfies the following properties: $\zeta(t)=1$ for every $t\in [0,1]$ and $|\zeta'(t)|+|\zeta''(t)|\le 1$ for every $t\in[0,+\infty)$.
 Therefore, as $n\to+\infty$, the sequence $(\zeta_n)_{n\in\mathbb N}$ converges pointwise and monotonically to $1$ over $[0,+\infty)$, and the sequences $(\zeta_n')_{n\in\mathbb N}$ and $(\zeta_n'')_{n\in\mathbb N}$ converge pointiwse to $0$ over $[0,+\infty)$.
 It is clear that $w_n\in\mathcal V_\eps^0$ for every $n\in\mathbb N$. Since $$w_n''(t,x)=w_0(x)\zeta_n''(t)+2w_1(x)\zeta_n'(t)+t\zeta_n''(t) w_1(x),$$ it follows by dominated convergence that
 \[
 \lim_{n\to+\infty}\frac{\eps^2}{2}\int_0^{+\infty}\int_{\R^N}e^{-t/\eps}|w_n''(t,x)|^2\,dx\,dt=0,
 \]
 and similarly
 \[
 \lim_{n\to+\infty} \int_0^{+\infty}e^{-t/\eps}W( w_n(t,\cdot))\,dt=\int_0^{+\infty} e^{-t/\eps}W(w_0+tw_1)\,dt.
 \]
 On the other hand, letting $C_0:=\int_{\R^N}w_0(x)\varphi(x)\,dx>0$, we have
 \[\begin{aligned}
 &\int_0^{+\infty}\int_{\R^N}e^{-t/\eps}F(\eps t,x,w_n(t,x))\,dx\,dt\\&\quad= \int_0^{+\infty}\int_{\R^N}e^{-t/\eps}\eta(\eps t)\varphi(x)(w_0(x)\zeta_n(t)+tw_1(x)\zeta_n(t))\,dx\,dt\ge C_0\int_0^{+\infty}e^{-t/\eps}\eta(\eps t)\zeta_n(t)\,dt
 \end{aligned}\]
 where the right hand side goes to $+\infty$ as $n\to+\infty$ by monotone convergence.
 %\[
 %\lim_{n\to+\infty}\int_0^{+\infty}e^{-t/\eps}\eta(\eps t)\xi_n(t)\,dt=+\infty.\]
We conclude that 
$
\lim_{n\to+\infty}\mathcal F_\eps(w_n)=-\infty.$
\end{proof}

\section{Proof of the main result}\label{proofsect}

\MMM
We are ready to give the proof of Theorem \ref{maintheorem}. The crucial  bounds in $W^{2,\infty}((0,T);H')$ for solutions $u_\eps$ of problem \eqref{min} are obtained by exploiting some arguments that we have introduced in \cite{MP}. \KKK

\begin{proofad1}
Let $\eps\in(0,\eps_F)$ and let $u_\eps$ be a solution to problem \eqref{min}. Let $w_\eps(t,x):=u_\eps(t/\eps,x)$, so that $w_\eps$ minimizes functional $\mathcal F_\eps$ over $\mathcal V_\eps^0$ (see Remark \ref{change}).
  Let $\varphi\in C^\infty_c((0,+\infty))$ and let $\xi$  be defined by \eqref{xi}. 
{
%Let us consider the Banach space $H$  endowed with the norm $\|\cdot\|_{H}:=\|\cdot\|_{L^{r}(\R^N)}+\|\cdot\|_{H^{1}(\R^N)}$ and let us denote its dual space  by $H'$.  
 Let $h\in H$.
For every $\delta\in\mathbb R$, the function $u_\eps(t,x)+\delta\xi(t)h(x)$ belongs to $\mathcal U_\eps^0$ and  \eqref{EL} holds thanks to Lemma \ref{lemmaEL}.
\KKK

\medskip

{\textbf {Step 1. \boldmath$W^{2,\infty}((0,T);H')$ bounds}.}
Since  by H\"older inequality we have
\[\begin{aligned}
&\left|\int_0^{+\infty}\xi(t)e^{-t}(DW(u_\eps(t,\cdot)),h)\,dt\right|\\&\qquad\le \int_0^{+\infty}\int_{\R^N}e^{-t}|\xi(t)|(|\nabla u_\eps(t,x)||\nabla h(x)|+|u_\eps(t,x)|^{r-1}|h(x)|)\,dx\,dt\\&\qquad
\le \|\nabla h\|_{L^2(\R^N)}\int_0^{+\infty}e^{-t}|\xi(t)|\,\|\nabla u_\eps(t,\cdot)\|_{L^2(\R^N)}\,dt\\&\qquad\quad+
\| h\|_{L^r(\R^N)}\int_0^{+\infty}e^{-t}|\xi(t)|\,\| u_\eps(t,\cdot)\|_{L^r(\R^N)}^{r-1}\,dt,
\end{aligned}\]
we deduce from \eqref{EL}, since $\xi''(t)e^{-t}=\varphi(t)$ and by also using \eqref{laplace}, that
\[\begin{aligned}
\left|\int_0^{\infty}\varphi(t)\int_{\R^N}u_\eps''(t,x)h(x)\,dx\,dt\right|&\le \eps^2\|\nabla h\|_{L^2(\R^N)}\int_0^{+\infty}e^{-t}|\xi(t)|\,\|\nabla u_\eps(t,\cdot)\|_{L^2(\R^N)}\,dt\\&\qquad+\eps^2
\| h\|_{L^r(\R^N)}\int_0^{+\infty}e^{-t}|\xi(t)|\,\| u_\eps(t,\cdot)\|_{L^r(\R^N)}^{r-1}\,dt\\&\qquad+\eps^2\|h\|_{L^2(\R^N)}\int_0^{+\infty}|\xi(t)|e^{-t}\|f(\eps t,\cdot)\|_{L^2(\R^N)}\,dt.
\end{aligned}\]
Therefore, Lemma \ref{lemma1}   entails 
\beeq\lab{estfond}
\left|\int_0^{\infty}\varphi(t)\int_{\R^N}u_\eps''(t,x)h(x)\,dx\,dt\right|\le \eps^2 \|h\|_{H}\int_0^{\infty}\Theta_\eps(\eps t)|\varphi(t)|\,dt,
\eneq
where \[\Theta_\eps(\tau):=(2+r)\left(\bar Q +(28+ 4\tau)\int_0^\tau Q_\eps(s)\,ds\right)+\sqrt{Q_\eps(\tau)},\]
and we notice that by exploiting Lemma \ref{Linfty} we get
\beeq\label{barkey}
\|\Theta_\eps\|_{L^\infty(0,\tau)}\le\sqrt{M_F(\tau)}+(2+r)(\bar Q+\tau(28+4\tau)M_F(\tau))=:\bar M_F(\tau)
\eneq
for every $\tau>0$, where we stress that $\bar M_F(\tau)$ depends only on $\tau,w_0,w_1$ and $F$, but not on $\eps$. \KKK

If $T>0$ and $\tilde\varphi\in C^\infty_c((0,T))$, then we may apply \eqref{estfond} with the test function $\varphi\in C^{\infty}_c((0,+\infty))$ given by $\varphi(t)=\tilde\varphi(\eps t)$. In this way, by changing variables we get 
%functions $\varphi$ such that $\|\varphi\|_{L^1(0,T)}\le 1$ we deduce
\[\begin{aligned}
&\eps\left|\int_0^{+\infty}\tilde\varphi(t)\int_{\R^N}w_\eps''(t,x)h(x)\,dx\,dt\right|=\left|\int_0^{+\infty}\varphi(t)u_\eps''(t,x)\,dx\,dt\right|\\&\qquad\qquad\le \eps^2\|h\|_H\int_0^{+\infty}\Theta_\eps(\eps t)|\tilde\varphi(\eps t)|\,dt=\eps \|h\|_H\int_0^T\Theta_\eps(t)|\tilde\varphi(t)|\,dt
\end{aligned}
\]
and therefore  since $\tilde\varphi$ is supported in $(0,T)$  we get by \eqref{barkey}
\[
\left|\int_0^T\tilde\varphi(t)\int_{\R^N}w_\eps''(t,x)h(x)\,dx\,dt\right|\le \|h\|_{H}\,\bar M_F(T)\,\int_0^T|\tilde\varphi(t)|\,dt.
\]
%%%%%%%%%%%%%%%%%
%\[
%\int_0^{\infty}\varphi(\eps t)\int_{\R^N}u_\eps''(t,x)h(x)\,dx\,dt= \eps\int_0^{\infty}\varphi(t)\int_{\R^N}w_\eps''(t,x)h(x)\,dx\,dt 
%\]
%thus by  \eqref{estfond} 
%\beeq\lab{estfond2}
%\left|\int_0^{\infty}\varphi(t)\int_{\R^N}w_\eps''(t,x)h(x)\,dx\,dt\right|\le  \|h\|_{H^1(\R^N)}\int_0^{\infty}c(t)|\varphi(t)|\,dt.
%\eneq
%Since $\varphi\in C^\infty_c((0,T))$ \eqref{estfond2} entails
%\beeq\lab{estfond3}
%\left|\int_0^{T}\varphi(t)\int_{\R^N}w_\eps''(t,x)h(x)\,dx\,dt\right|\le \|h\|_{H^1(\R^N)}c(T)\int_0^{T}|\varphi(t)|\,dt.
%\eneq
By taking the supremum among all test functions $\tilde\varphi\in C^\infty_c((0,T))$ such that $\int_0^T|\tilde\varphi(t)|\,dt\le 1$ we get 
\beeq\label{2derivate}
\left \| \int_{\R^N}w_\eps''(t,x)h(x)\,dx\right \|_{L^\infty(0,T)}\le \bar M_F(T)\,\|h\|_{H}.
\eneq
Since \PPP $w_\eps\in H^2_{loc}([0,+\infty);L^2(\R^N))$ and $w_\eps'(0,\cdot)=w_1(\cdot)$, for every $t\ge 0$ we have \KKK
 \beeq\label{fundamental}
 \int_{\R^N}w_\eps'(t,x)h(x)\,dx=\int_{\R^N}w_1(x)h(x)\,dx+\int_0^t\int_{\mathbb R^N}w_\eps''(\tau,x)h(x)\,dx\,d\tau
 \eneq
 and we get as a consequence
 \beeq\label{1derivata}
\left\|\int_{\R^N}w_\eps'(\cdot,x)h(x)\,dx\right\|_{L^\infty(0,T)}\le (\|w_1\|_{L^2(\R^N)}+T\bar M_F(T))\,\|h\|_{H},
\eneq
and similarly
 \beeq\label{0derivate}
\left\|\int_{\R^N}w_\eps(\cdot,x)h(x)\,dx\right\|_{L^\infty(0,T)}\le (\|w_0\|_{L^2(\R^N)}+T\|w_1\|_{L^2(\R^N)}+T^2\bar M_F(T))\,\|h\|_{H}.
\eneq
The estimates \eqref{2derivate}-\eqref{1derivata}-\eqref{0derivate} hold true for every $\eps\in(0,\eps_F)$, 
therefore the family $(w_\eps)_{\eps\in(0,\eps_F)}$ is uniformly bounded in \PPP $W^{2,\infty}((0,T);H')$, \KKK and as such along a suitable vanishing sequence $(\eps_n)_{n\in\mathbb N}\subset(0,\eps_F)$ it converges weakly in $H^2((0,T);H')$ to some $w\in W^{2,\infty}((0,T);H')$, and thus the convergence holds in the sense of distributions on $(0,T)\times\R^N$. %The a-priori estimate of second time derivates allow to keep the initial data in the limit:
Moreover,
 \PPP since the family $(\int_{\R^N}w_{\eps}'(\cdot,x)h(x)\,dx)_{\eps\in(0,\eps_F)}$ is bounded in $H^1((0,T))$, along a further subsequence it converges uniformly on $[0,T]$, therefore we can pass to the limit in \eqref{fundamental} and get 
  \[
 \langle w'(t,\cdot),h(\cdot)\rangle=\int_{\R^N}w_1(x)h(x)\,dx+\int_0^t\langle w''(\tau,\cdot),h(\cdot)\rangle\,d\tau\qquad\mbox{for every $t\in[0,T]$},
 \]
 where $\langle\cdot,\cdot\rangle$ denotes the duality between $H'$ and $H$.
  By evaluating at $t=0$ we identify $w'(0,\cdot)$ with $w_1$ in $H'$, since $h\in H$ is arbitrary. Therefore  $w'(0,x)=w_1(x)$  a.e. in $\mathbb R^N$. Similarly we get $w(0,x)=w_0(x)$ a.e. in $\R^N$. \KKK

%since the embedding $H^1((0,T);H')\hookrightarrow C^0([0,T];H')$ \MMM is compact\KKK, both $w_{\eps_n}$ and $w'_{\eps_n}$ converge (to $w$ and $w'$, respectively)  in $C^0([0,T];H')$ so that $w_{\eps_n}(0,\cdot)\to w(0,\cdot)$ and $w'_{\eps_n}(0,\cdot)\to w'(0,\cdot)$ in $H'$,   therefore $w(0,x)=w_0(x)$ and $w'(0,x)=w_1(x)$ \MMM a.e. in $\mathbb R^N$. \KKK

\medskip

{\textbf{Step 2. Further a-priori estimates.}}
%Further a-priori estimates are obtained by taking advantage of Lemma \ref{appendixlemma}. 
%indeed, if $K_\eps(t):=\frac1{2\eps^2}\|u_\eps'(t,\cdot)\|^2_{L^2(\mathbb R^N)}$ we have $K_\eps(t/\eps)=\frac1{2}\|w_\eps'(t,\cdot)\|^2_{L^2(\mathbb R^N)}$ and therefore 
\MMM Since $w_\eps(t,x):=u_\eps(t/\eps,x)$,
from Lemma \ref{appendixlemma} we deduce that for every $t> 0$  
\begin{equation}\label{inL2}\frac1{2}\|w_\eps'(t,\cdot)\|^2_{L^2(\mathbb R^N)}\le Q +(28t+4t^2)\,\|Q_\eps\|_{L^\infty(0,t)}\le Q +(28t+4t^2)\,M_F(t) \end{equation}
where $Q$ is  the constant therein, and where Lemma \ref{Linfty} has been invoked as well,  so that the right hand side of \eqref{inL2}  depends only on $F,w_0,w_1$ and $t$, but not on $\eps$. \KKK
For every $T>0$ we may write 
\[\begin{aligned}
&\int_0^T\int_{\R^N}|w_\eps(t,x)|^2\,dx\,dt\le 2T\|w_0\|^2_{L^2(\R^N)}+\int_0^T\int_{\R^N}|w_\eps(t,x)-w_0(x)|^2\,dx\,dt\\&\qquad=
2T\|w_0\|^2_{L^2(\R^N)}+\int_0^T\int_{\R^N}\left|\int_0^tw'_\eps(s,x)\,ds\right|^2\,dx\,dt. 
\end{aligned}\]
By applying \eqref{inL2} and Jensen inequality we deduce that there exists a constant $C_T$ (only depending on $w_0,w_1, F$ and $T$) such that
\begin{equation}\label{inL22}
\int_0^T\int_{\R^N}|w_\eps(t,x)|^2\,dx\,dt\le C_T.
\end{equation}
By \eqref{inL2} and \eqref{inL22}, the family $(w_\eps)_{\eps\in(0,\eps_F)}$ is uniformly bounded in $H^1((0,T);L^2(\mathbb R^N))$.
 Moreover,  \KKK
  by arguing as in the proof  of \cite[Formula (2.9)]{ST2},  for every $s\in [0,1]$ we get (by taking into account \MMM estimate \eqref{basicenergy} in Lemma \ref{exist})
 \beeq\begin{aligned}\displaystyle\int_s^{s+1}W(u_\eps(t,\cdot))\,dt&\le e^2\int_0^{\infty}e^{-t}W(w_\eps(t,\cdot))\,dt\\
 &\displaystyle\le\int_0^{+\infty}e^{-t}\left\{\frac1{2\eps^2}\|u_{\eps}''(t)\|_{L^2(\R^N)}^2+W(u_{\eps}(t))\right\}\,dt\le \bar C
 \end{aligned}
 \eneq
 and by recalling Lemma \ref{appendixlemma} and Lemma \ref{Linfty}, since $\min_{x\in[1,2]} xe^{-x}= 2e^{-2}$,   for every $s>1$ we get \MMM
 \[\begin{aligned} \int_s^{s+1}W(u_\eps(t,\cdot))\,dt&\le \frac{e^2}2\int_s^{s+1}(t-s+1)e^{-(t-s+1)}W(w_\eps(t,\cdot))\,dt\\
 &\displaystyle\le \frac{e^2}2\int_{s-1}^{+\infty}(t-s+1)e^{-(t-s+1)}W(w_\eps(t,\cdot))\,dt\\&\le\frac{Qe^2}2+\frac{e^2}2(28+4\eps s)\,\int_0^{\eps s}Q_\eps(\tau)\,d\tau\le \frac{Qe^2}2+\frac{e^2}2(28+4\eps s)\, M_F(\eps s)\,\eps s, \end{aligned}
 \]
so that for every $s\ge 0$ we obtain the estimate
$$\int_s^{s+1}W(u_\eps(t,\cdot))\,dt\le \max\left\{\bar C,\,\frac{Qe^2}2+\frac{e^2}2(28+4\eps s)\, M_F(\eps s)\,\eps s \right\}.
$$
Since $u_\eps(t,\cdot)=w_\eps(\eps t,\cdot)$, by changing variables we get for every $\tau\ge0$
\beeq\lab{estW}\int_\tau^{\tau+\eps}W(w_\eps(t,\cdot))\,dt\le \eps\, \max\left\{\bar C,\,\frac{Qe^2}2+\frac{e^2}2(28+4\tau)\,\tau\, M_F(\tau) \right\}.
\eneq
Therefore, if  $T \ge \eps$, by covering $[0, T]$ with consecutive intervals of lenght $\eps$ and using \eqref{estW} in each interval one obtains
\[\int_0^{T}W(w_\eps(t,\cdot))\,dt\le \tilde C_T,\]
where $\tilde C_T$ is a constant depending only on $w_0$, $w_1$, $F$ and $T$. In particular, also recalling \eqref{inL22},
the family $(w_\eps)_{\eps\in(0,\eps_F)}$ is uniformly bounded in $L^r((0,T)\times\R^N)$ and in $L^2((0,T);H^1(\R^N))$.}

%Combined with the previous a-priori estimate in   $H^1((0,T);L^2(\mathbb R^N))$, this allows to conclude that the family $(w_\eps)_{\eps\in(0,1)}$
%is uniformly bounded in $W^{1,2\wedge p}((0,T)\times B_R)$, for every $T>0$ and every ball $B_R$ in $\R^N$.

\medskip

{\textbf{Step 3. Convergence to nonlinear wave equation.}}
 Summing up, we get the existence of a vanishing sequence $(\eps_n)_{n\in\mathbb N}\subset (0,\eps_F)$, such that $w_{\eps_n}$ converges weakly in $H^{1}((0,T)\times B_R)$ and strongly in $L^q((0,T)\times B_R)$ for every $T>0$, every $R>0$ and every  $q\in [2,r)$ ($q=2$ if $r\le 2$), weakly in $L^2((0,+\infty)\times\R^N)$, and pointwise a.e in $(0,+\infty)\times \R^N$ to some $w$ that belongs to $\in H^1((0,T);L^2(\R^N))\cap H^2((0,T);H')$ for every $T>0$.
\KKK
 Let now $\eta\in C^{\infty}_c((0,T)\times\R^N)$. \KKK Then by setting $\psi_\eps(t,x):=\eta (t,x)e^{t/\eps}$ and by taking into account the first order minimality condition satisfied by the minimizer $w_\eps$ of functional $\mathcal F_\eps$, which is the analogous of \eqref{EL}, we have
\[\begin{aligned}
&\displaystyle -\int_0^{T}\int_{\R^N}{w'_\eps}(t,x)\left(\eps^{2}\eta'''(t,x)+ 2\eps\eta''(t,x)+\eta'(t,x)\right)\,dx\,dt\\&\quad=-\eps^{2}\int_0^{T}\int_{\R^N}{w_\eps'}(t,x) \left(\psi''_\eps(t,x)\,e^{-t/\eps}\right)'\,dx\,dt\\
&\quad\displaystyle =\eps^{2}\int_0^{T}\int_{\R^N}{w''_\eps}(t,x) \psi_\eps''(t)\,e^{-t/\eps} \,dx\,dt=\eps^{2}\int_0^{+\infty}\int_{\R^N}{w''_\eps}(t,x) \psi_\eps''(t,x)\,e^{-t/\eps}\,dt\\
&\quad =-\int_0^{+\infty}(D W(w_\eps(t,\cdot)),\psi_\eps(t,\cdot))\,e^{-t/\eps}\,dt+\int_0^{+\infty}\int_{\R^N}
\psi_\eps(t,x)\,G(t,x, w_\eps(t,x))\,e^{-t/\eps}\,dx\,dt
\\&\quad=-\int_0^{T}\int_{\R^N}\nabla w_\eps(t,x)\cdot\nabla \eta(t,x)\,dx\,dt
-\int_0^{T}\int_{\R^N}| w_\eps(t,x)|^{r-2}\, w_\eps(t,x) \eta(t,x)\,dx\,dt
\\&\qquad\quad+\int_0^{T}\int_{\R^N}
G(t,x,w_\eps(t,x))\,\eta(t,x)\,dx\,dt.
%\\
%&\qquad\qquad\displaystyle= \int_0^{+\infty}\mathbf f_h(t)\pphi_h(t) e^{-ht}\,dt=\int_0^{T}\mathbf f_h(t)\xxi(t)\,dt.
\end{aligned}\]
% Let now $\eta\in C^{\infty}(\R)$ with  $\spt \eta\subset (0,T)$ and $\beta\in C^\infty_c(\R^N)$. \KKK Then by setting $\psi_\eps(t,x):=\beta(x)\eta (t)e^{t/\eps}$ and by taking into account the first order minimality condition satisfied by the minimizer $w_\eps$ of functional $\mathcal F_\eps$ we have
%\[\begin{aligned}
%&\displaystyle -\int_0^{T}\int_{\R^N}{w'_\eps}(t,x)\left(\eps^{2}\eta'''(t)+ 2\eps\eta''(t)+\eta'(t)\right)\beta(x)\,dx\,dt\\&\quad\qquad=-\eps^{2}\int_0^{T}\int_{\R^N}{w_\eps'}(t,x) \left(\psi''_\eps(t,x)\,e^{-t/\eps}\right)'\,dx\,dt\\
%&\quad\qquad\displaystyle =\eps^{2}\int_0^{T}\int_{\R^N}{w''_\eps}(t,x) \psi_\eps''(t)\,e^{-t/\eps} \,dx\,dt=\eps^{2}\int_0^{+\infty}\int_{\R^N}{w''_\eps}(t,x) \psi_\eps''(t,x)\,e^{-t/\eps}\,dt\\
%&\quad\qquad =-\int_0^{+\infty}\int_{\R^N}\nabla w_\eps(t,x)\cdot\nabla\psi_\eps(t,x)e^{-t/\eps}\,dx\,dt+\int_0^{+\infty}\int_{\R^N}
%\psi_\eps(t,x)f(t,x)e^{-t/\eps}\,dx\,dt
%\\&\quad\qquad=-\int_0^{T}\int_{\R^N}\nabla w_\eps(t,x)\cdot\nabla \beta(x)\eta(t)\,dx\,dt+\int_0^{T}\int_{\R^N}
%f(t,x)\beta(x)\eta(t)\,dx\,dt.
%%\\
%%&\qquad\qquad\displaystyle= \int_0^{+\infty}\mathbf f_h(t)\pphi_h(t) e^{-ht}\,dt=\int_0^{T}\mathbf f_h(t)\xxi(t)\,dt.
%\end{aligned}\]
Since $w_{\eps_n}\to w$ in the sense of distributions and pointwise a.e. on $(0,T)\times\mathbb R^N$, we pass to the limit as $n\to+\infty$ to obtain for every $\eta\in C^\infty_c((0,T)\times\R^N)$
\[\begin{aligned}
& \int_0^{T}\int_{R^N}(-{w'}(t,x)\eta'(t,x)+\nabla w(t,x)\cdot\nabla\eta(t,x)+|w(t,x)|^{r-2}w(t,x)\eta(t,x))\,dx\,dt\\&\qquad=\int_0^{T}\int_{\R^N}
G(t,x, w(t,x))\eta(t,x)\,dx\,dt,
\end{aligned}\]
i.e., 
\[
w''=\Delta w-|w|^{r-2}w+G(t,x,w)\quad \mbox{in}\ \mathcal D'((0,T)\times\R^N),
\]
%&w(0,x)=w_0(x),\ w'(0,x)=w_1(x)\ \mbox{a.e. in}\ \R^N .\\
%\end{array}\right.
%\eneq
and the arbitrariness of $T$  allows to conclude.
In order to pass to the limit, we have also used dominated convergence, taking advantage of \eqref{C1}-\eqref{f}-\eqref{hyp0},  for the term involving $G$. On the other hand,  we have used the strong convergence of $w_{\eps_n}$ in $L^q((0,T)\times B_R)$ for every $1<q<r$ ($q=2$ if $r\le 2$) for the term involving  $| w_{\eps_n}|^{r-2}\, w_{\eps_n}$.
\KKK

\medskip

{\textbf{Step 4. Energy inequality}.} Here the proof is the very same of  \cite{TT1}.
\MMM
By \eqref{approxenergy}  and \cite[Lemma 6.1]{ST2} we obtain that for every $a>0$, every $\delta\in(0,1)$ and every $T\ge 0$
\[
\int_0^{a\delta}s e^{-s}\,ds\int_{T+\delta a}^{T+a}W(u_\eps(s,\cdot))\,ds\le \int_T^{T+a}\left(E_\eps(s)-\frac1{2\eps^2}\|u_\eps'(s,\cdot)\|^2_{L^2(\mathbb R^N)}\right)\,ds,
\]
which after scaling and changing variables, recalling that $u_\eps(t,x)=w_\eps(\eps t,x)$, entails
\begin{equation*}
\int_0^{a \delta/\eps}te^{-t}\,dt\int_{T+\delta a}^{T+a}W(w_\eps(t,\cdot))\,dt+\frac12\int_T^{T+a}\|w_\eps'(t,\cdot)\|^2_{L^2(\R^N)}\,dt\le \int_T^{T+a}E_\eps(t/\eps)\,dt.
\end{equation*}
%By taking \eqref{tt1} into account we get for every $\beta>1$
%\[
%\begin{aligned}
%&\int_0^{a \delta/\eps}te^{-t}\,dt\int_{T+\delta a}^{T+a}W(w_\eps(t,\cdot))\,dt+\frac12\int_T^{T+a}\|w_\eps'(t,\cdot)\|^2_{L^2(\R^N)}\,dt\\&\qquad\le \int_T^{T+a}
%\left(\sqrt{E_\eps(0)}+\left(\sqrt\eps\tilde C_\beta+\sqrt{t\beta/2}\right)\left(\int_0^t Q_\eps(s)\,ds\right)^{1/2}\right)\,dt
%\end{aligned}
%\]
We take the limit as $\eps\to 0$: in the previous steps we have 
 already established the weak convergence, along a suitable vanishing sequence, of $w_{\eps}$ to $w$ in $L^2((0,T_0);H^1(\R^N))$, in $ H^1((0,T_0);L^2(\R^N))$ and in $ L^r((0,T_0)\times\R^N)$, for every $T_0>0$, thus by weak lower semicontinuity of $L^p$ norms we deduce
\[
\int_{T+\delta a}^{T+a}W(w(t,\cdot))\,dt+\frac12\int_T^{T+a}\|w'(t,\cdot)\|^2_{L^2(\R^N)}\,dt\le \int_T^{T+a}\limsup_{\eps\to 0}E_\eps(t/\eps)\,dt
\]
where we have used Fatou Lemma in the right hand side. 
%However, by \eqref{tt1} we have
%\[
%\sqrt{E_\eps(t/\eps)}\le \sqrt{E_\eps(0)}+\left(\sqrt\eps\tilde C_\beta+\sqrt{t\beta/2}\right)\left(\int_0^t Q_\eps(s)\,ds\right)^\frac12
%\]
%for every $\beta>1$, and we recall that by Lemma \ref{Lebesgue} there holds
%\[
%\lim_{\eps\to 0} \int_0^tQ_\eps(s)\,ds=\int_0^t \|f(s,\cdot)\|^2_{L^2(\R^N)}\,ds,
%\]
%so that by exploiting \eqref{newinappendixlemma} we deduce
%\[
%\limsup_{\eps\to0}\sqrt{E_\eps(t/\eps)}\le \left(\frac12\|w_1\|_{L^2(\R^N)}^{2}+W(w_0)\right)^{\frac12}+\sqrt{t\beta/2}\,\left(\int_0^t\|f(s,\cdot)\|^2_{L^2(\R^N)}\,ds\right)^{\frac12}.
%\]
By taking the limit as $\delta\to 0$ and by taking \eqref{newinappendixlemma} into account we get
\[\begin{aligned}
&\frac1a\int_{T }^{T+a}W(w(t,\cdot))\,dt+\frac1{2a}\int_T^{T+a}\|w'(t,\cdot)\|^2_{L^2(\R^N)}\,dt\\&\qquad\le \frac1a\int_T^{T+a}\left(\left(\frac12\|w_1\|_{L^2(\R^N)}^{2}+W(w_0)\right)^\frac12+\sqrt{t/2}\,\left(\int_0^t\|f(s,\cdot)\|^2_{L^2(\R^N)}\,ds\right)^{\frac12}\right)^2\,dt
\end{aligned}\]
for every $T\ge 0$ and every $a>0$. By  letting $a\to 0$ we obtain
% deduce that for a.e. $t>0$
%\[
%W(w(t,\cdot))+\frac12\|w_1\|^2_{L^2(\R^N)}\le \left(\left(\frac12\|w_1\|_{L^2(\R^N)}^{2}+W(w_0)\right)^\frac12+\sqrt{\frac t2}\,\left(\int_0^t\|f(s,\cdot)\|^2_{L^2(\R^N)}\,ds\right)^{\frac12}\right)^2
%\]
%thus proving
 \eqref{energyinequality}.
\end{proofad1}

\KKK

\section{Appendix}

The main purposes of this appendix are to prove
the following  technical lemma that is required in order to adapt some key estimates from \cite{TT1} to our framework, and then to prove Lemma \ref{appendixlemma}
% that there exists $c_0>0$ such that $|g(t)|\le c_0 t$ and $|g'(t)|\le c_0$ for every $t\ge 0$.

\begin{lemma}\label{longlemma} Let $\eps\in(0,\eps_F)$ and let $u_\eps\in\mathcal U_\eps^0$.
Let $g:[0,+\infty)\to \mathbb R$ be a $C^1$ function. Suppose   that there exists $c_0>0$ such that $|g(t)|\le c_0 t$  for every $t\ge 0$ and $|g'(t)|\le c_0$ for every $t\ge 0$. Let \MMM$0<\delta_0<\tfrac{\eps_F-\eps}{c_0\eps_F}$\KKK. For every $\delta\in(-\delta_0,\delta_0)$, let $g_\delta:[0,+\infty)\to[0,+\infty)$ be defined by $g_\delta(t)=t-\delta g(t)$. Then, the mapping 
\begin{equation}\label{deltamap}
(-\delta_0,\delta_0)\ni\delta\mapsto \int_{0}^{+\infty}\int_{\mathbb R^N} e^{-t} F(\eps t,x, u_{\eps}(g_\delta(t),x))\,dx\,dt
\end{equation} 
is differentiable at $\delta=0$ with derivative
\begin{equation}\label{derivata}
%\frac{d}{d\delta}{\Bigg{|}}_{\delta=0}\int_{0}^{+\infty}\int_{\mathbb R^N} e^{-t} F(\eps t,x, u_{\eps}(g_\delta(t),x))\,dx\,dt=
-\int_0^{+\infty}\int_{\mathbb R^N}g(t)\,e^{-t}\,\frac{\partial F}{\partial v}(\eps t,x, u_\eps(t,x))\,u_\eps'(t,x)\,dx\,dt,
\end{equation}
and the mapping
\[
(-\delta_0,\delta_0)\ni\delta\mapsto \int_{0}^{+\infty}\int_{\mathbb R^N} e^{-t} F(\eps t,x, u_{\eps}(g_\delta(t),x)+t\delta\eps g'(0)w_1(x))\,dx\,dt
\]
is differentiable at $\delta=0$ with derivative
\begin{equation}\label{derivata2}
%\frac{d}{d\delta}{\Bigg{|}}_{\delta=0}\int_{0}^{+\infty}\int_{\mathbb R^N} e^{-t} F(\eps t,x, u_{\eps}(g_\delta(t),x))\,dx\,dt=
-\int_0^{+\infty}\int_{\mathbb R^N}\,e^{-t}\,\frac{\partial F}{\partial v}(\eps t,x, u_\eps(t,x))\,(g(t)\,u_\eps'(t,x)-t\eps g'(0)\,w_1)\,dx\,dt.
\end{equation}

\end{lemma}
\begin{proof} We notice that $e^{-t} F(\eps t,x, u_{\eps}(t,x))\in L^1((0,+\infty)\times\mathbb R^N)$ is a consequence of  $u_\eps\in\mathcal U_\eps^0$ as seen in the proof of Lemma \ref{exist}, see \eqref{below}.
We first check that $e^{-t} F(\eps t,x, u_{\eps}(g_\delta(t),x))\in L^1((0,+\infty)\times\mathbb R^N)$ as well, for every $\delta\in(-\delta_0,\delta_0)$, so that the mapping \eqref{deltamap} is well defined. 
In order to obtain this property, a preliminary remark is that
$$
\int_{0}^{+\infty}\int_{\R^N}e^{-t}F(\eps t,x,w_0(x))\,dx\,dt\le C_F+K_F^*\|w_0\|_{L^2(\R^N)},
$$
as a direct consequence of the assumptions on $F$, using  \eqref{lip}, Jensen inequality and Lemma \ref{quicklemma}.
 We notice that, for every $\delta\in(-\delta_0,\delta_0)$, $g_\delta$ is a $C^1$ strictly increasing bijection with $C^1$  inverse. Moreover, since $u_\eps\in H^2((0,T);L^2(\R^N))$ for every $T>0$, we have that $s\mapsto u_\eps(s,x)$ is absolutely continuous on $[0,T]$ for a.e. $x\in\mathbb R^N$. Therefore by also invoking \eqref{lip} we get
 \[\begin{aligned}
&\int_{0}^{+\infty}\int_{\mathbb R^N} e^{-t} |F(\eps t,x, u_{\eps}(g_\delta(t),x))|\,dx\,dt\\&\qquad=
\int_{0}^{+\infty}\int_{\mathbb R^N} e^{-t} |F(\eps t,x, u_{\eps}(g_\delta(t),x))-F(\eps t,x,w_0(x))+F(\eps t,x,w_0(x))|\,dx\,dt\\&\qquad\le
C_F+ K_F^*\|w_0\|_{L^2(\R^N)}+ \int_{0}^{+\infty}\int_{\mathbb R^N} e^{-t} f(\eps t,x)| u_{\eps}(g_\delta(t),x)-u_\eps(0,x)|\,dx\,dt\\&\qquad\le
C_F+ K_F^*\|w_0\|_{L^2(\R^N)}+ \int_{0}^{+\infty}\int_{\mathbb R^N} e^{-t} f(\eps t,x)\int_{0}^{g_\delta(t)}|u_\eps'(s,x)|\,ds\,dx\,dt
\end{aligned}\]
and we need to check that the last integral is finite:  by two applications of Cauchy-Schwarz inequality, it is bounded above by
\begin{equation}\label{quattro}\begin{aligned}
& \int_{0}^{+\infty}\int_{\mathbb R^N} e^{-t+{g_\delta(t)/2}} f(\eps t,x)\int_{0}^{g_\delta(t)}e^{-s/2}|u_\eps'(s,x)|\,ds\,dx\,dt\\&\qquad\le
 \int_{0}^{+\infty} e^{-t+{g_\delta(t)/2}} \|f(\eps t,\cdot)\|_{L^2(\R^N)}\left(\int_{\R^N}\left(\int_{0}^{g_\delta(t)}e^{-s/2}|u_\eps'(s,x)|\,ds\right)^2\,dx\right)^\frac12\,dt\\&\qquad\le
 \int_0^{+\infty}e^{-t+{g_\delta(t)/2}}\|f(\eps t,\cdot)\|_{L^2(\R^N)}\sqrt{g_\delta(t)}\left(\int_{\R^N}\int_0^{+\infty} e^{-s}|u_\eps'(s,x)|^2\,ds\,dx\right)^\frac12\,dt,
\end{aligned}\end{equation}
where the interior integral over $(0,+\infty)\times\R^N$ is finite due to \eqref{'0} since $u_\eps\in \mathcal U_\eps^0$, 
%and
%$e^{-t+{g_\delta(t)/2}}\sqrt{g_\delta(t)}$ is in $L^1(0,+\infty)$ for every $\delta\in(-\delta_0,\delta_0)$ since $|g_\delta(t)|\le (1+\delta_0c_0)t $ and $\delta_0c_0<1$. 
\MMM and where we notice that  $g_\delta(t)=t-\delta g(t)\le(1+\delta_0 c_0)t\le 2t$ for every $t\ge 0$, implying
\begin{equation}\label{basicfinite}\begin{aligned}
 &\int_0^{+\infty}e^{-t+{g_\delta(t)/2}}\|f(\eps t,\cdot)\|_{L^2(\R^N)}\sqrt{g_\delta(t)}\,dt\\&\qquad\le \eps^{-3/2}\int_0^{+\infty}\sqrt{2t}\,e^{-\frac{1-\delta_0c_0}{2\eps}t}\,\|f(t,\cdot)\|_{L^2(\R^N)}\,dt<+\infty\end{aligned}
\end{equation}
due to the condition \eqref{laplace}, since the restriction $\delta_0<\tfrac{\eps_F-\eps}{c_0\eps_F}$ yields $\tfrac{1-\delta_0c_0}{2\eps}>\tfrac1{2\eps_F}$. \KKK
In a similar way, we can prove that $e^{-t} F(\eps t,x, u_{\eps}(g_\delta(t),x)+t\delta\eps g'(0)w_1(x))\in L^1((0,+\infty)\times\mathbb R^N)$. Indeed, by  \eqref{lip} we have for a.e. $(t,x)\in(0,+\infty)\times\R^N$
\[
|F(\eps t,x, u_{\eps}(g_\delta(t),x)+t\delta\eps g'(0)w_1(x))-F(\eps t,x, u_{\eps}(g_\delta(t),x))|\le |t\delta\eps g'(0)w_1(x)|f(\eps t,x)
\]
so that we are  left to prove that $te^{-t}f(\eps t,x)w_1(x)$ is in $L^1((0,+\infty)\times\mathbb R^N)$, which immediately follows from Lemma \ref{quicklemma} and $w_1\in L^2(\R^N)$ since by Young inequality 
\begin{equation}\label{adp}\begin{aligned}
\int_0^{+\infty}\int_{\R^N}te^{-t}f(\eps t,x)|w_1(x)|\,dx\,dt&\le
\frac12\int_0^{+\infty}\int_{\R^N}te^{-t}(f(\eps t,x)^2+|w_1(x)|^2)\,dx\,dt\\
&\le\frac12 K_F^*(1)+\frac12\|w_1\|_{L^2(\R^N)}^2.
\end{aligned}
%\|w_1\|_{L^2(\R^N)} \int_0^{+\infty}\frac{te^{-t/\eps}}{\eps^2}\|f(t,\cdot)\|_{L^2(\R^N)}\,dt<+\infty.
\end{equation}

% We notice that, for every $\delta\in(-\delta_0,\delta_0)$, $g_\delta$ is a $C^1$ strictly increasing bijection with $C^1$  inverse. Moreover, since $u_\eps\in H^1((0,T)\times\mathbb R^N)$, we have that $u_\eps$ is absolutely continuous on almost every line. Therefore
% \[
% \begin{aligned}
% &\left|\int_0^{+\infty}\int_{\mathbb R^N}e^{-t}f(\eps t,x)(u_\eps(g_{\delta}(t),x)-w_0(x))\,dx\,dt\right|\le \int_0^{+\infty}\int_{\mathbb R^N}e^{-t}|f(\eps t,x)|\left(\int_0^{g_\delta(t)}|u_\eps'(s,x)|\,ds\right)\,dx\,dt...
% \end{aligned}
% \]
  We also notice that the integrand in \eqref{derivata} is indeed in $L^1((0,+\infty)\times\R^N)$, since \eqref{laplace} yields
 \begin{equation}\label{adr}\begin{aligned}
 &\int_0^{+\infty}\int_{\mathbb R^N}|g(t)|\,e^{-t}\,\left|\frac{\partial F}{\partial v}(\eps t,x, u_\eps(t,x))\right|\,|u_\eps'(t,x)|\,dx\,dt\\&\qquad
\le \int_0^{+\infty}\int_{\mathbb R^N}|g(t)|\,e^{-t}\,f(\eps t,x)\,u_\eps'(t,x)\,dx\,dt\\&\qquad\le \left( \int_0^{+\infty}\int_{\mathbb R^N}e^{-t}|g(t)|^2f(\eps t,x)^2\,dx\,dt\right)^\frac12\left( \int_0^{+\infty}\int_{\mathbb R^N}e^{-t}|u_\eps'(t,x)|^2\,dx\,dt\right)^\frac12\\&\qquad
\le c_0\sqrt{K_F^*(2)}\, \left( \int_0^{+\infty}\int_{\mathbb R^N}e^{-t}|u_\eps'(t,x)|^2\,dx\,dt\right)^\frac12
%\\&\qquad\le
% K_F \left( \int_0^{+\infty}\int_{\mathbb R^N}c_0^2t^2e^{-t}\,dx\,dt\right)^\frac12\left( \int_0^{+\infty}\int_{\mathbb R^N}e^{-t}|u_\eps'(t,x)|^2\,dx\,dt\right)^\frac12,
\end{aligned} \end{equation}
where we have used Lemma \ref{quicklemma} and where
the integral in the right hand side is finite thanks to $u_\eps\in \mathcal U_\eps^0$ (by recalling \eqref{'0} as before). Besides,  the second term in the integrand of \eqref{derivata2} is in $L^1((0,+\infty)\times\R^N)$ due to  \eqref{adp}.

Let us prove that the map \eqref{deltamap}  is differentiable at $\delta=0$.
 Since for $u_\eps$ the (classical) time derivative $u_\eps'(t,x)$ exists at a.e. $(t,x)\in(0,+\infty)\times\mathbb R^N$, we have
\begin{equation*}
\lim_{\delta\to 0}\frac{u_\eps(g_\delta(t),x)-u_\eps(t,x)}{\delta} =-g(t) u'_\eps(t,x) \qquad\mbox{ for a.e. $(t,x)\in(0,+\infty)\times\mathbb R^N$.}
\end{equation*}
Therefore, thanks to the fact that, for a.e. $(t,x)\in(0,+\infty)\times\R^N$, the map $v\mapsto F(t,x,v)$ is assumed to be \MMM $C^1(\R)$\KKK, see \eqref{C1},  we deduce
\begin{equation}\label{pointwise}
\lim_{\delta\to 0}\frac{F(\eps t,x,u_\eps(g_\delta(t),x))-F(\eps t,x,u_\eps(t,x))}{\delta} =-g(t)\,\frac{\partial F}{\partial v}(\eps t,x,u_\eps(t,x))\, u'_\eps(t,x) 
\end{equation}
 for a.e. $(t,x)\in(0,+\infty)\times\mathbb R^N$ by classical chain rule.
We next prove that for a.e. $t\in(0,+\infty)$
\begin{equation}\label{inspace}\begin{aligned}
&\lim_{\delta\to 0}\int_{\mathbb R^N} e^{-t}  \frac{F(\eps t,x,u_{\eps}(g_\delta(t),x))-F(\eps t,x,u_\eps(t,x))}{\delta}\,dx\\&\qquad\qquad=-\int_{\mathbb R^N}g(t)e^{-t}\,\frac{\partial F}{\partial v}(\eps t,x,u_\eps(t,x))\,u_\eps'(t,x)\,dx,\end{aligned}
\end{equation}
Indeed, after having  fixed  $t$ (outside a null set of times), pointwise a.e. convergence of the integrands follows from \eqref{pointwise}; on the other hand, still for the same fixed $t$ a dominating function is obtained by taking advantage of the continuous inclusion of $H^{2}((0,T);L^2(\mathbb R^N))$ into $W^{1,\infty}((0,T);L^2(\mathbb R^N))$ holding for every $T>0$, so that by taking large enough $T$ we have
\[
|u_\eps(g_\delta(t),x)-u_\eps(t,x)|\le \ell_T(x)|g_\delta(t)-t|=|\delta g(t)|\ell_T(x) \qquad\mbox{ for a.e. $x\in\mathbb R^N$}
\]
for some suitable nonnegative $\ell_T\in L^2(\mathbb R^N)$,
thus by \eqref{lip} for every $\delta\in(-\delta_0,0)\cup(0,\delta_0)$ we get the estimate
\[
\frac{|F(\eps t,x,u_\eps(g_\delta(t),x))-F(\eps t,x,u_\eps(t,x))|}{|\delta|} e^{-t}\le e^{-t}|g(t)||f(\eps t,x)|\ell_T(x)\qquad\mbox{ for a.e. $x\in\mathbb R^N$},
\]
where the right hand side is integrable in $\mathbb R^N$ since $\ell_T\in L^2(\mathbb R^N)$ and $f(\eps t,\cdot)\in L^2(\mathbb R^N)$ by \eqref{laplace}. 
Having shown that pointwsie a.e. $t\in(0,+\infty)$ convergence \eqref{inspace}, we are left to find a dominating function for passing to the limit through the time integral. We take advantage of the fact that $e^{-t/2}u_\eps\in H^2((0,+\infty);L^2(\mathbb R^N))\hookrightarrow W^{1,\infty}((0,+\infty);L^2(\mathbb R^N))$, so that both $e^{-t/2}u_\eps$ and $e^{-t/2}u_\eps'$ belong to $L^\infty((0,+\infty);L^2(\mathbb R^N))$. Letting $A_\eps:=\mathrm{esssup}_{t>0}\|e^{-t/2}u_\eps'(t,\cdot)\|_{L^2(\mathbb R^N)}$  we obtain with \eqref{lip} and with similar arguments as in \eqref{quattro} the estimate
\[
\begin{aligned}
&\left|\int_{\mathbb R^N} e^{-t}  \frac{F(\eps t,x,u_{\eps}(g_\delta(t),x))-F(\eps t,x,u_\eps(t,x))}{\delta}\,dx\right|\le\int_{\mathbb R^N} \frac{e^{-t}}{|\delta|} f(\eps t,x) \left|\int_{t}^{g_\delta(t)}u_\eps'(s,x)\,ds\right|dx\\
&\qquad\le \frac{e^{-t}e^{(t\vee g_\delta(t))/2}}{|\delta|}\int_{\mathbb R^N}f(\eps t,x)\int_{t\wedge g_\delta(t)}^{t\vee g_\delta(t)}e^{-s/2}|u'_\eps(s,x)|\,ds\,dx\\
&\qquad\le\frac1{|\delta|} e^{(-1+|\delta|c_0)t/2}\int_{\mathbb R^N}f(\eps t,x)|t-g_\delta(t)|^{\frac12}\left(\int_{t\wedge g_\delta(t)}^{t\vee g_\delta(t)}|e^{-s/2}u'_\eps(s,x)|^2\,ds\right)^{\frac12}\,dx\\
&\qquad\le \frac1{|\delta|} e^{(-1+|\delta|c_0)t/2}|t-g_\delta(t)|^{\frac12}\left(\int_{\mathbb R^N}|f(\eps t,x)|^2\,dx\right)^{\frac12}\left(\int_{\mathbb R^N}\int_{t\wedge g_\delta(t)}^{t\vee g_{\delta(t)}}|e^{-s/2}u_\eps'(s,x)|^2\,ds\right)^{\frac12}\\
&\qquad\le  \frac{A_\eps}{|\delta|} e^{(-1+|\delta|c_0)t/2}|t-g_\delta(t)|\,\|f(\eps t,\cdot)\|_{L^2(\R^N)}
\le  {A_\eps}c_0\, t e^{(-1+\delta_0c_0)t/2} \|f(\eps t,\cdot)\|_{L^2(\R^N)},
\end{aligned}
\] 
for every $\delta\in (-\delta_0,0)\cup(0,\delta_0)$,
where the right hand side is integrable on $(0,+\infty)$ by the same argument used for proving \eqref{basicfinite}. By dominated convergence, we conclude that the map \eqref{deltamap} is differentiable at $\delta=0$ with derivative given by \eqref{derivata}.

Very similar arguments allow to obtain the last statement.
\end{proof}

The following two lemmas are a slight extension of the results from \cite{ST2,TT1} and they  are written by borrowing some of the notation therein.
%Indeed, they are proven in \eqref{TT1} in the case $F(t,x,v)=b(t,x)\,v$ under additional assumptions  on $b$. 

\begin{lemma}\label{borrow} Let $\eps\in(0,\eps_F)$ and let $u_\eps$ be a solution to problem \eqref{min}.
 Let 
\[
K_\eps(t):=\frac1{2\eps^2}\|u_\eps'(t,\cdot)\|^2_{L^2(\mathbb R^N)}, \quad
%\qquad W_\eps(t):=\frac1p\|\nabla u_\eps(t,\cdot)\|^p_{L^p(\mathbb R^N)}+\frac1r\| u_\eps(t,\cdot)\|^r_{L^r(\mathbb R^N)},\] 
D_\eps(t):=\frac1{2\eps^2}\|u_\eps''(t,\cdot)\|^2_{L^2(\mathbb R^N)},\quad L_\eps(t):=D_\eps(t)+W(u_\eps(t,\cdot)).
\]
Let $g:[0,+\infty)\to \mathbb R$ be a $C^{1,1}$ nonnegative function such that  $g(0)=0$ and such that $g(t)$ is affine for large $t$.
Then, there holds
\begin{equation}\label{deltaderivative}
\begin{aligned}
&\int_0^{+\infty}\!\!\!\!\!\!e^{-t}\left[(g'(t)-g(t))L_\eps(t)-4g'(t)D_\eps(t)-g''(t)K_\eps'(t)+\eps tg'(0)  (DW (u_\eps(t,\cdot)), w_1)\right]dt\\&\qquad=
-\int_0^{+\infty}\int_{\mathbb R^N}\,e^{-t}\,\frac{\partial F}{\partial v}(\eps t,x, u_\eps(t,x))\,(g(t)\,u_\eps'(t,x)-t\eps g'(0)\,w_1(x))\,dx\,dt.
%-\int_0^{+\infty}\int_{\mathbb R^N}e^{-t}g(t)f(\eps t,x)u_\eps'(t,x)\,dx\,dt+\eps g'(0)\int_{0}^{+\infty}\int_{\mathbb R^N}te^{-t}f(\eps t,x)w_1(x)\,dx\,dt. %\\&\qquad-\eps g'(0)\int_0^{+\infty}te^{-t}\langle\Delta u_\eps(t,\cdot),w_1(\cdot)\rangle\,dt
\end{aligned}
\end{equation}
\end{lemma}
\begin{proof} The proof for $F\equiv 0$ is given in \cite[Proposition 4.4, Corollary 4.5]{ST2}. It is obtained by considering the perturbation $u_{\delta,\eps}$ of  $u_\eps$, defined as in Lemma \ref{longlemma} by $u_{\eps,\delta}(t,x):=u_\eps(g_\delta(t),x)+t\delta\eps g'(0)w_1$ for $\delta\in(-\delta_0,\delta_0)$ and suitably small $\delta_0$.  Indeed, it is shown in \cite{ST2} that in the case $F\equiv 0$ the derivative of $\delta\mapsto J(u_{\eps,\delta})$ at $\delta=0$ is given by the left hand side of \eqref{deltaderivative}, all the terms in the integral therein being indeed integrable on $(0,+\infty)$, so that \eqref{deltaderivative}  follows by the minimality of $u_\eps$ since it is clear that $u_{\eps,\delta}$ is an admissible competitor.
The more general result including a nonnull $F$ is therefore deduced by invoking
%In presence of a forcing term $f$,  the $\delta$-derivative of $J_\eps(u_{\eps,\delta})$ is obtained by adding the derivative of $\delta\mapsto \int_0^{+\infty}\int_{\mathbb R^N}e^{-t}u_{\eps,\delta}(t,x)f(\eps t,x)\,dx\,dt$.
%At $\delta=0$, the $\delta$-derivative of such a term is given by the right hand side of \eqref{deltaderivative} as shown in \cite[Proposition 4.4]{TT1} under different assumptions on $f$, and therefore we need to extend the computations from \cite{TT1} in order to match our assumptions on $f$. 
 Lemma \ref{longlemma}.
\end{proof}

%For any nonnegative, measurable function $\zeta$ defined on $[0,+\infty)$ with $\int_0^{+\infty}(1+s)e^{-s}\zeta(s)\,ds<+\infty$, we let 
%\[
%\mathcal A \zeta (t):=\int_t^{+\infty}e^{-(s-t)}\zeta(s)\,ds,\qquad \mathcal A\mathcal A \zeta(t)=\mathcal A^2 \zeta(t):=\int_{t}^{+\infty}e^{-(s-t)}(s-t)\zeta(s)\,ds.
%\]
%so that there hold the a.e. equalities 
%\begin{equation}\label{A2}(\mathcal A\zeta)'=\mathcal A\zeta-\zeta\qquad\mbox{and}\qquad (\mathcal A^2\zeta)'=\mathcal A^2\zeta-\mathcal A\zeta.
%\end{equation}
%With this notation, we can obtain suitable estimates from Lemma \ref{longlemma} by applying it to particular choices of the function $g$ therein. Indeed, we have
\begin{lemma} For every $\eps\in(0,\eps_F)$, let  $u_\eps$ be a solution to problem \eqref{min}.
%By using the same assumptions and notation of {\rm Lemma \ref{borrow}}, 
%there holds
%\[
%\mathcal A^2 D_\eps(0)=\frac1{2\eps^2}\int_{0}^{+\infty}se^{-s}\|u_\eps''(s,\cdot)\|^2_{L^2(\mathbb R^N)}\,ds<+\infty,
%\]
%\begin{equation}\label{in0}\begin{aligned}
%&\int_0^{+\infty}te^{-t}L_\eps(t)\,dt-\int_{0}^{+\infty}e^{-t}L_\eps(t)\,dt+4\int_{0}^{+\infty}e^{-t}D_\eps(t)\,dt\\&\qquad=\int_0^{+\infty}\int_{\mathbb R^N}se^{-s}\frac{\partial F}{\partial v}(\eps s,x,u_\eps(s,x))\,u_\eps'(s,x)\,dx\,ds
%+\eps\int_0^{+\infty}se^{-s}( DW (u_\eps(s,\cdot)), w_1)\,ds\\&\qquad\quad-\eps\int_{0}^{+\infty}\int_{\mathbb R^N}se^{-s}\frac{\partial F}{\partial v}(\eps s,x, u_\eps(s,x))w_1(x)\,dx\,ds
%\end{aligned}\end{equation}
There exists a constant $\bar B$, not depending on $\eps$ but  only depending on $w_0,w_1$ and $F$,  such that 
\begin{equation}\label{A2W0}
 \int_0^{+\infty}se^{-s}\,W(u_\eps(s,\cdot))\,ds%= \int_{0}^{+\infty}se^{-s}\|\nabla u_\eps(s,\cdot)\|^2_{L^2(\mathbb R^N)}\,ds
\le \bar B.
\end{equation}
\MMM
In the limit as $\eps\to 0$ we also have the estimate
\beeq\label{specific}
\limsup_{\eps\to 0}  \int_0^{+\infty}se^{-s}\,W(u_\eps(s,\cdot))\,ds\le W(w_0).
\eneq
\KKK
Moreover, by using the same notation of {\rm Lemma \ref{borrow}}, for a.e. $\tau>0$ there holds
\begin{equation}\label{dirac}\begin{aligned}
%\mathcal A^2 L_\eps(\tau)+ 4\mathcal A D_\eps(\tau)-\mathcal A L_\eps(\tau)=\int_{\tau}^{+\infty}e^{-(s-\tau)}(s-\tau)\langle f(\eps s,\cdot),u_\eps'(s,\cdot)\rangle_{L^2(\mathbb R^N)}\,ds- K'_\eps(\tau).
&\int_\tau^{+\infty}e^{-(t-\tau)}L_\eps(t)\,dt-\int_{\tau}^{+\infty}(t-\tau)e^{-(t-\tau)}L_\eps(t)\,dt-4\int_\tau^{+\infty}e^{-(t-\tau)}D_\eps(t)\,dt\\&\qquad=K_\eps'(\tau)-\int_{\tau}^{+\infty}\int_{\mathbb R^N}(t-\tau)e^{-(t-\tau)}\,\frac{\partial F}{\partial v}(\eps t,x, u_\eps(t,x))\,u'_\eps(t,x)\,dx\,dt.
\end{aligned}
\end{equation}
\end{lemma} 
\begin{proof}
By taking $g(t)=t$ in Lemma \ref{longlemma}
 we get
%\[
%\mathcal A^2 D_\eps(0)=\frac1{2\eps^2}\int_{0}^{+\infty}se^{-s}\|u_\eps''(s,\cdot)\|^2_{L^2(\mathbb R^N)}\,ds<+\infty,
%\]
\begin{equation*}\begin{aligned}
&\int_0^{+\infty}te^{-t}L_\eps(t)\,dt-\int_{0}^{+\infty}e^{-t}L_\eps(t)\,dt+4\int_{0}^{+\infty}e^{-t}D_\eps(t)\,dt\\&\qquad=\int_0^{+\infty}\int_{\mathbb R^N}se^{-s}\frac{\partial F}{\partial v}(\eps s,x,u_\eps(s,x))\,u_\eps'(s,x)\,dx\,ds
+\eps\int_0^{+\infty}se^{-s}( DW (u_\eps(s,\cdot)), w_1)\,ds\\&\qquad\quad-\eps\int_{0}^{+\infty}\int_{\mathbb R^N}se^{-s}\frac{\partial F}{\partial v}(\eps s,x, u_\eps(s,x))w_1(x)\,dx\,ds,
\end{aligned}\end{equation*}
therefore, since %$\mathcal A^2 W_\eps(0)\le \mathcal A^2 L_\eps(0)$, 
$
\int_0^{+\infty}se^{-s}W(u_\eps(s,\cdot))\,ds\le \int_{0}^{+\infty}se^{-s}L_\eps(s)\,ds,
$
 from  \eqref{laplace} we obtain 
\beeq\label{manyterms}\begin{aligned}
&%\mathcal A^2W_\eps(0)
\int_0^{+\infty}se^{-s}W(u_\eps(s,\cdot))\,ds\le
 %\mathcal A L_\eps(0)
 \int_{0}^{+\infty}e^{-s}L_\eps(s)\,ds
 +\int_0^{+\infty}\int_{\mathbb R^N}se^{-s}f(\eps s,x)\,|u_\eps'(s,x)|\,dx\,ds
\\&\qquad+\eps\int_0^{+\infty}se^{-s}|( DW(u_\eps(s,\cdot)), w_1)|\,dt+\eps\int_{0}^{+\infty}\int_{\mathbb R^N}se^{-s}f(\eps s,x)\,|w_1(x)|\,dx\,ds,
\end{aligned}\end{equation}
and we may estimate the terms in the right hand side uniformly in $\eps$. Indeed, $ \int_{0}^{+\infty}e^{-s}L_\eps(s)\,ds$ is the left hand side of \eqref{basicenergy} and therefore it gets estimated by $\bar C$ as therein.  We treat $\int_0^{+\infty}\int_{\mathbb R^N}se^{-s}f(\eps s,x)\,|u_\eps'(s,x)|\,dx\,ds$ by means of \eqref{adr} with $g(t)=t$, and the  integral in the right hand side of \eqref{adr} gets estimated by a constant 
%We have 
%\[\begin{aligned}
%&\int_0^{+\infty}\int_{\mathbb R^N}se^{-s}f(\eps s,x)\,|u_\eps'(s,x)|\,dx\,ds\le \int_0^{+\infty}\int_{\mathbb R^N}s^2e^{-s}|f(\eps s,x)|^2u\,dx\,ds\\&\qquad+\int_0^{+\infty}\int_{\mathbb R^N}e^{-s}|u_\eps'(s,x)|^2\,dx\,ds\le 5\bar C +\int_0^{+\infty}\int_{\mathbb R^N}e^{-s}|u_\eps'(s,x)|^2\,dx\,ds
%\end{aligned}\]
%where the last integral is estimated by a constant that depends
 depending only on $w_0,w_1$ and $F$ by means of \eqref{'0} and \eqref{basicenergy}. Similarly we use \eqref{adp} for estimating $\int_{0}^{+\infty}\int_{\mathbb R^N}se^{-s}f(\eps s,x)|w_1(x)|\,dx\,ds$.  Moreover, by Young inequality
\[\begin{aligned}
&\int_0^{+\infty}se^{-s}|(DW(u_\eps(s,\cdot)), w_1)|\,ds\\&\;\;\;\le
%\frac12\int_0^{+\infty}\int_{\mathbb R^N} e^{-s}|\nabla u_\eps(s,x)|^2\,ds\,dx+\|w_1\|_{H^1(\mathbb R^N)}^2
\int_0^{+\infty}\int_{\R^N}se^{-2}[|\nabla u_\eps(s,x)||\nabla w_1(x)|+|u_\eps(s,x)|^{r-1}|w_1|]\,dx\,ds\\&\;\;\;\le
\int_0^{+\infty}\int_{\R^N}e^{-s}\left(\frac{s^2}2 |\nabla w_1(x)|^2+\frac{s^r}r|w_1(x)|^r+\frac{1}{2}|\nabla u_\eps(s,x)|^2+\frac{r-1}{r}|u_\eps(s,x)|^r\right)\,dx\,ds\\
&\;\;\;\le \|\nabla w_1\|_{L^2(\R^N)}^2+\frac{\Gamma(r+1)}r\|w_1\|^r_{L^r(\R^N)}+r\int_0^{+\infty}e^{-s}W(u_\eps(s,\cdot))\,ds
\end{aligned}\]
and we conclude by estimating the right hand side by means of \eqref{basicenergy}, thus getting \eqref{A2W0}. \MMM By the above estimate and by Lemma \ref{quicklemma}, it is also clear that the last two terms of \eqref{manyterms} are vanishing as $\eps\to0$. The second term in the right hand side of \eqref{manyterms} vanishes as well as $\eps\to 0$, because by applying Cauchy-Schwarz inequality along with  \eqref{stimacon0},  \eqref{basicenergy} and Lemma \ref{quicklemma} we see that
\[\begin{aligned}
&\int_0^{+\infty}\int_{\mathbb R^N}se^{-s}f(\eps s,x)\,|u_\eps'(s,x)|\,dx\,ds\le \sqrt{K_F^*(2)}\left(
\int_0^{+\infty}e^{-s}\|u_\eps'(s,\cdot)\|^2_{L^2(\R^N)}\right)^\frac12\\&\qquad\le
\sqrt{K_F^*(2)}\left(2\eps^2\|w_1\|^2_{L^2(\R^N)}+4
\int_0^{+\infty}e^{-s}\|u_\eps''(s,\cdot)\|^2_{L^2(\R^N)}\right)^\frac12\\&\qquad\le \sqrt{K_F^*(2)}\left(2\eps^2\|w_1\|^2_{L^2(\R^N)}+8\eps^2\bar C\right)^{\frac12}.
\end{aligned}\]
 Therefore, by taking the limit as $\eps\to0$ in \eqref{manyterms}, by recalling the definition of $L_\eps$ from  Lemma \ref{borrow} and by applying Lemma \ref{newlemma}, we deduce  \eqref{specific}.  \KKK
 
%and similarly one may invoke \eqrer{adp} for treating 
%\[
%\int_{0}^{+\infty}\int_{\mathbb R^N}se^{-s}f(\eps s,x)|w_1(x)|\,dx\,ds\le \frac12\int_{0}^{+\infty}\int_{\mathbb R^N}s^2e^{-s}|f(\eps s,x)|^2\,dx\,ds+\frac12\|w_1\|_{L^2(\mathbb R^N)}
%\]
%so that both terms get estimated by a constant that depends only on $w_0,w_1$ and $C$  by means of \eqref{hyp} and \eqref{basicenergy} 
%again. 

%Finally \eqref{dirac} is obtained from \eqref{deltaderivative} by choosing $g(t)=(t-\tau)_+$, whose second derivative is the Dirac mass at $\tau$.
%Although $g$ is not $C^{1,1}$, this can be justified by approximating $g$ with $C^{1,1}$ functions $(g_n)_{n\in\mathbb N}$ that converge pointwise to $g$ as $n\to+\infty$, by means of the argument in \cite[Corollary 4.7]{ST2}. 
Finally \eqref{dirac} is obtained from \eqref{deltaderivative} by means of the argument in \cite[Corollary 4.7]{ST2}, where such a  formula is proved in the case $F\equiv 0$. Indeed
 choosing $g=g_n$ in Lemma \ref{longlemma}, where $g_n$ is the sequence of $C^{1,1}$ functions from \cite[Corollary 4.7]{ST2}, converging pointwise to $g(t)=(t-\tau)_+$ as $n\to+\infty$, in the case $F\equiv 0$ one obtains \eqref{dirac} as shown therein. If $F$ is nonnull, the terms involving $F$ in \eqref{deltaderivative} pass to the limit as $n\to+\infty$ by dominated convergence.
 % whose second derivative is the Dirac mass at $\tau$.
%Although $g$ is not $C^{1,1}$, this can be justified by approximating $g$ with $C^{1,1}$ functions $(g_n)_{n\in\mathbb N}$ that converge pointwise to $g$ as $n\to+\infty$, by means of the argument in \cite[Corollary 4.7]{ST2}. %In particular, by means of such approximation it is easy to pass to the limit as $n\to+\inftu$
\end{proof}

Now we can give the proof of Lemma \ref{appendixlemma}. The proof is
 based on the arguments in \cite[Proposition 5.4]{TT1}, along with Lemma \ref{Lebesgue}. 
 
 \begin{proofad2} 
We define $K_\eps$ and $D_\eps$ as in Lemma \ref{borrow}. 
A computation shows that  $$E'_\eps(t)=K_\eps'(t)+\int_t^{+\infty}(s-t)e^{-(s-t)}W(u_\eps(s,\cdot))\,ds-\int_t^{+\infty}e^{-(s-t)} W(u_\eps(s,\cdot))\,ds$$ for a.e. $t>0$, which, in combination with 
\eqref{dirac} and \eqref{laplace}, yields
\[\begin{aligned}
E_\eps'(t)&=-3\int_t^{+\infty}e^{-(s-t)} D_\eps(s)\,ds-\int_t^{+\infty}(s-t)e^{-(s-t)} D_\eps(s)\,ds\\&\qquad+\int_{t}^{+\infty}\int_{\R^N}e^{-(s-t)}(s-t)\,\frac{\partial F}{\partial v}(\eps s,x, u_\eps(s,x))\,u_\eps'(s,x)\,dx\,ds,\\&\le
-3\int_t^{+\infty}e^{-(s-t)} D_\eps(s)\,ds-\int_t^{+\infty}(s-t)e^{-(s-t)} D_\eps(s)\,ds\\&\qquad+\int_{t}^{+\infty}\int_{\R^N}e^{-(s-t)}(s-t)\,f(\eps s,x)\,|u_\eps'(s,x)|\,dx\,ds,
\end{aligned}\]
 thus  by Cauchy-Schwarz inequality we get
 \[\begin{aligned}
E_\eps'(t)&\le -3\int_t^{+\infty}e^{-(s-t)} D_\eps(s)\,ds-\int_t^{+\infty}(s-t)e^{-(s-t)} D_\eps(s)\,ds\\&\;+\left(\int_t^{+\infty}(s-t)e^{-(s-t)}\|f(\eps s,\cdot)\|_{L^2(\mathbb R^N)}^2\,ds\right)^\frac12\left(\int_t^{+\infty}(s-t)e^{-(s-t)}\|u_\eps'(s,\cdot)\|_{L^2(\mathbb R^N)}^2\,ds\right)^{\frac12}
\end{aligned}\]
% \[
%E_\eps'(t)\le -3\mathcal A D_\eps(t)-\mathcal A^2 D_\eps(t)+\left(\mathcal A^2\|f(\eps t,\cdot)\|_{L^2(\mathbb R^N)}^2\right)^\frac12\left(\mathcal A^2\|u_\eps'(s,\cdot)\|_{L^2(\mathbb R^N)}^2\right)^{\frac12}
%\]
for a.e. $t>0$. By means  of a suitable version of the Gronwall lemma, it is shown in the proof of \cite[Proposition 5.4]{TT1} that the latter estimate implies that for every $t\ge0$ \MMM and every $\beta>1$ 
\begin{equation*}
\sqrt{E_\eps(t)}\le\sqrt{E_\eps(0)}+\sqrt\eps (\tilde C_\beta+\sqrt {t\beta/2})\left(\eps\int_0^t\left(\int_s^{+\infty}(\tau-s)e^{-(\tau-s)}\|f(\eps \tau,\cdot)\|_{L^2(\mathbb R^N)}^2\,d\tau\right)\,ds\right)^\frac12,
\end{equation*}
where $\tilde C_\beta$ is a suitable positive constant, only depending on $\beta$, such that $\tilde C_2=\sqrt{2\sqrt2/(\sqrt2-1)}$.
 By recalling the definition of $Q_\eps$ from \eqref{Theta}, the latter estimate reads
\begin{equation}\label{tt1}\sqrt{E_\eps(t)}\le \sqrt{E_\eps(0)}+\sqrt{\eps}(\tilde C_\beta+\sqrt{t\beta/2})\left(\eps\int_0^t Q_\eps(\eps s)\,ds\right)^{\frac12}.
\end{equation}
%so that \eqref{tt1} is proved.
%\begin{equation}\label{tt1}
%\sqrt{E_\eps(t)}\le\sqrt{E_\eps(0)}+\eps (C_0+\sqrt t)\left(\int_0^t\mathcal A^2\|f(\eps s,\cdot)\|_{L^2(\mathbb R^N)}^2\,ds\right)^\frac12,
%\end{equation}
 \KKK Since $E_\eps(0)=K_\eps(0)+\int_0^{+\infty}se^{-s} W(u_\eps(s,\cdot))\,ds$, since $K_\eps(t)=\tfrac1{2\eps^2}\|u_\eps'(t,\cdot)\|^2_{L^2(\mathbb R^N)}$ and $u_\eps'(0,\cdot)=\eps w_1$, we see that $K_\eps(0)=\frac12\|w_1\|^2_{L^2(\R^N)}$ and we obtain  $E_\eps(0)\le\tfrac12\|w_1\|^2_{L^2(\mathbb R^N)}+\bar B=:Q$  where $\bar B$ is the constant appearing in \eqref{A2W0}. \MMM
Therefore from \eqref{tt1} \MMM with $\beta=2$ \KKK we obtain \eqref{tt}, since $\eps<\eps_F<1$. 
%On the other hand, from \eqref{specific} we get \eqref{newinappendixlemma}.
 \KKK
 \MMM Moreover, still by \eqref{tt1} we have
\[
\sqrt{E_\eps(t/\eps)}\le \sqrt{\frac{1}{2}\|w_1\|^2_{L^2(\R^N)}+\int_0^{+\infty} se^{-s}W(u_\eps(s,\cdot))\,ds}+\left(\sqrt\eps\tilde C_\beta+\sqrt{t\beta/2}\right)\left(\int_0^t Q_\eps(s)\,ds\right)^\frac12
\]
for every $t\ge 0$ and for every $\beta>1$, and we recall that by Lemma \ref{Lebesgue} there holds
\[
\lim_{\eps\to 0} \int_0^tQ_\eps(s)\,ds=\int_0^t \|f(s,\cdot)\|^2_{L^2(\R^N)}\,ds,
\]
so that by exploiting \eqref{specific} we deduce
\[
\limsup_{\eps\to0}\sqrt{E_\eps(t/\eps)}\le \left(\frac12\|w_1\|_{L^2(\R^N)}^{2}+W(w_0)\right)^{\frac12}+\sqrt{t\beta/2}\,\left(\int_0^t\|f(s,\cdot)\|^2_{L^2(\R^N)}\,ds\right)^{\frac12}.
\]
By taking the limit as $\beta\to 1$ we get \eqref{newinappendixlemma}.
 \KKK
%{\color{blue}Therefore, by taking into account of \eqref{hyp2b} and $2\eps <1$, 
%%and by setting $Q_1:= 2Q_0^2 ,\ Q_2:=\max\{8K_{F}^2C_0^2, 8K_F^2\},$
% \eqref{tt1} entails
%\begin{equation}\label{tt2}\begin{array}{ll}
%&\displaystyle{E_\eps(t)}\le 2Q_0^2+2\eps^2 (C_0+\sqrt t)^2\int_0^t\,ds\int_s^{+\infty}(\tau-s)e^{(s-\tau)}\|f(\eps\tau,\cdot)\|_{2}^{2}\,d\tau\le\\
%&\\
%&\displaystyle\le 2Q_0^2+ 2(C_0+\sqrt {\eps t})^2\,e^{\eps t}\int_0^{\eps t}\,d\sigma\int_\sigma^{+\infty}\eps^{-2}(\theta-\sigma)e^{(1-\frac{1}{\eps})(\theta-\sigma)}e^{-\theta}\|f(\theta,\cdot)\|_{2}^{2}\,d\tau\\
%\end{array}
%\end{equation}
%whence by setting 
%\[ g_\eps(\xi):=-\eps^{-2}\xi\,e^{(-1+\frac{1}{\eps})\xi}{\mathbf 1}_{(-\infty,0]}(\xi);\quad h(\xi):=e^{-\xi}\|f(\xi,\cdot)\|_{2}^{2}{\mathbf 1}_{[0,+\infty)}(\xi)\]
%we get by Young's convolution inequality
%\beeq\begin{array}{ll}&\displaystyle
%E_\eps(t)\le  2Q_0^2+ 2(C_0+\sqrt {\eps t})^2\,e^{\eps t}\|g_\eps*h\|_1\le\\
%&\\
%&\displaystyle\le 2Q_0^2+ 2(C_0+\sqrt {\eps t})^2\,e^{\eps t}\|g_\eps\|_1\|h\|_1=\\
%&\\
%&\displaystyle 2Q_0^2+ 2(C_0+\sqrt {\eps t})^2\,e^{\eps t}(1-\eps)^{-2}\,K_F\\
%&\\
%&\displaystyle\le 2Q_0^2+ 8(C_0+\sqrt {\eps t})^2\,K_F\,e^{\eps t}
%\end{array}
%\eneq
%and \eqref{tt} follows by choosing $Q_1:= 2Q_0^2+ 16C_0^2,\ Q_2:= 16K_F.$}
\end{proofad2}

\subsection*{Acknowledgements} 
The authors acknowledge support from the MIUR-PRIN  project  No 2017TEXA3H.  EM is supported by the Istituto Nazionale di Alta Matematica (INDAM) project CUP$_-$E55F22000270001.
%The authors are members of the
%GNAMPA group of the Istituto Nazionale di Alta Matematica (INdAM).

%


\begin{thebibliography}{99}
\bibitem{A}{P. Avramidou}, \textit{Convolution operators induced by approximate identities and pointwise convergence in $L^p(\mathbb R)$  spaces}. Proceedings of the AMS 133  (2004), no. 1, 175--184.
\bibitem{B}{H.S. Bear}, \textit{Approximate identities and pointwise convergence}. Pacific J. of Math. 81 (1979), 17--27.
\bibitem{DG}{E. De Giorgi}, \textit{Conjectures concerning some evolution problems. A celebration of John F. Nash, jr}. Duke Math. J. 81 (1996), 61-100.
%\bibitem{G}{M.Gurtin}, {\textit{An Introduction to Continuum Mechanics}}, {Springer, 1999}.
% \bibitem{LS} M. Liero, U. Stefanelli, {\textit{A new minimum principle for Lagrangian mechanics}}, J. Nonlinear Sci. 23 (2) (2013) 179--204.
 %\bibitem{LS2}  M. Liero, U. Stefanelli, \textit{Weighted inertia-dissipation-energy functionals for semilinear equations}, Boll. Unione Mat. Ital. (9) 6 (1) (2013) 1--27.
 
 \bibitem{DG2}{E. De Giorgi}, {Selected Papers}. 
Springer-Verlag, New York, 2006, (edited by L. Ambrosio, G. Dal Maso, M. Forti, M. Miranda, and S. Spagnolo).

%\bibitem{L1} G. Lebeau, {\textit Non linear optic and supercritical wave equation}, Bull. Soc. Roy. Sci. Li\`ege 70 (2001), 267--
%306 (2002)

\bibitem{K} A. Kundu, \textit{Shape Changing and Accelerating Solitons in the Integrable Variable Mass Sine-Gordon Model}, Phys. Rev. Lett. 99, 154101 (2007). 

\bibitem{L} G. Lebeau, \textit{Perte de r\'egularit\'e pour les \'equations d'ondes sur-critiques}, Bull. Soc. Math. France 133 (2005), no. 1, 
145--157.

 \bibitem{LS} M. Liero, U. Stefanelli, {\textit{A new minimum principle for Lagrangian mechanics}}, J. Nonlinear Sci. 23 (2) (2013), 179--204.
 \bibitem{LS2}  M. Liero, U. Stefanelli, \textit{Weighted inertia-dissipation-energy functionals for semilinear equations}, Boll. Unione Mat. Ital. (9) 6  (2013), no. 1, 1--27.
%\bibitem{M} V. Maz'ya,  Sobolev Spaces with Applications to Elliptic Partial Differential Equations. Second revised
%and augmented edition. Grundlehren der Mathematischen Wissenschaften [Fundamental Principles of
%Mathematical Sciences], vol. 342. Springer, Heidelberg (2011).
\bibitem{MP} E. Mainini, D. Percivale, \textit{Newton’s second law as limit of variational problems},  Adv. Cont. Discr. Mod.  2023, 20 (2023).
\bibitem {ST}{E. Serra, P. Tilli}, {\textit { Nonlinear wave equation as limits of convex minimization problems: proof of a conjecture by De Giorgi}, }Annals of Math.  175 (2012),  1551--1574.
\bibitem{ST2}{E. Serra, P. Tilli}, {\textit {A minimization approach to hyperbolic Cauchy problems}}. J. Eur. Math. Soc. 18 (2016), no. 9, 2019--2044.

\bibitem{IMM} S. Ibrahim, M. Majdoub, N. Masmoudi,   \textit{Well- and ill-posedness issues for energy
supercritical waves}, Anal. PDE 4 (2011), no. 2, 341--367.

\bibitem {S} {U. Stefanelli}, \textit{The De Giorgi conjecture on elliptic regularization}, Math. Models Methods Appl. Sci. 21 (2011), 1377--1394.

\bibitem{Sh} {J. Shatah and M. Struwe}, Geometric Wave Equations, Courant Lect. Notes Math. 2, New York University Courant Institute of Mathematical Sciences, New York, 1998.

\bibitem{Str}{M. Struwe}, \textit{Semilinear wave equations},  Bull.  Amer. Math. Soc. 26 (1992), no.1, 53--85.

\bibitem{Str2}{M. Struwe}, \textit{On Uniqueness and Stability for Supercritical Nonlinear Wave and Schr\"odinger Equations}, Int.Math.
Res. Not. 2006 (2006), 76737

\bibitem{T}{T. Tao}, \textit{Finite-time blow up for
a supercritical defocusing nonlinear wave system}, Anal. PDE 9 (2016), no. 8, 1999--2030. 

\bibitem{TT1} L. Tentarelli, P. Tilli, {\textit{De Giorgi’s approach to hyperbolic Cauchy problems: the case of nonhomogeneous equations}}, Comm. Partial Differential Equations 43 (4) (2018), 677--698.

\bibitem{TT2} L. Tentarelli, P. Tilli, {\textit{An existence result for dissipative nonhomogeneous hyperbolic equations via a minimization approach}},
J. Differential Equations 266 (2019), 5185--5208.
\end{thebibliography}
\end{document}